\documentclass[reqno,a4paper,11 pt]{amsart}
\usepackage[numbers]{natbib}
\usepackage[a4paper,top=2cm,bottom=2cm,left=2cm,right=2cm,bindingoffset=5mm]{geometry}

\newtheorem{theorem}{Theorem}
\newtheorem{proposition}{Proposition}
\newtheorem {lemma}{Lemma}
\newtheorem {example}{Example}
\newtheorem {definition}{Definition}

\newtheorem{remark}{Remark}
\newtheorem{hypothesis}[theorem]{Hypothesis}

\newenvironment{system}
{\left\lbrace\begin{array}{@{}l@{}}}
{\end{array}\right.}

\def\ds{\begin{displaystyle}}
\def\eds{\end{displaystyle}}
\def\dis{\displaystyle }
\def\<{\langle }
\def\>{\rangle }

\newcommand{\R}{\mathbb{R}}
\newcommand{\E}{\mathbb{E}}
\newcommand{\mP}{\mathbb{P}}
\newcommand{\Fcal}{\mathcal{F}}     
\newcommand{\Lcal}{\mathcal{L}}
\newcommand{\calm}{\mathcal{M}}
\newcommand{\calg}{\mathcal{G}}
\newcommand{\calb}{\mathcal{B}}
\newcommand{\calu}{\mathcal{U}}

\def\R{\mathbb R}

\def\E{\mathbb E}
\def\P{\mathbb P}

\title[SMP for SPDEs with delay]{Stochastic maximum principle for SPDEs with delay}


\thanks{Financial support from the grant
MIUR-PRIN 2010-11 ``Evolution differential problems: deterministic
and stochastic approaches and their interactions''
is gratefully acknowledged. The authors have been
supported by the Gruppo Nazionale per l'Analisi Matematica,
la Probabilit\`a e le loro Applicazioni (GNAMPA)
of the Istituto Nazionale di Alta Matematica (INdAM).}

\begin{document}

\author{Giuseppina Guatteri}
\address[G. Guatteri]{Dipartimento di Matematica, Politecnico di Milano. via Bonardi 9, 20133 Milano, Italia}
\email{giuseppina.guatteri@polimi.it}

\author{Federica Masiero}
\address[F. Masiero]{Dipartimento di Matematica e Applicazioni, Universit\`a di Milano-Bicocca. via Cozzi 55, 20125 Milano, Italia}
\email{federica.masiero@unimib.it}

\author{Carlo Orrieri}
\address[C. Orrieri]{Dipartimento di Matematica, Universit\`a di Pavia. via Ferrata 1, 27100 Pavia, Italia}
\email{carlo.orrieri01@ateneopv.it}

\begin{abstract}
In this paper we develop necessary conditions for optimality, in the form of the Pontryagin maximum principle, for the optimal control problem of a
class of infinite dimensional evolution equations with delay in the state. In the cost functional 
we allow the final cost to depend on the history of the state. To treat such kind of cost functionals
we introduce a new form of anticipated backward stochastic differential equations which plays the role of dual equation associated to the 
control problem. 
\end{abstract}

\maketitle

\section{Introduction}

In this paper we study the stochastic maximum principle (SMP) for controlled stochastic partial differential equations with delay
in the state. Maximum principle for evolution equations have been proved in \cite{hu1990maximum}, see also
the recent papers \cite{fuhrman2012stochastic}, \cite{fuhrman2013stochastic}, \cite{du2013maximum}, and 
\cite{fuhrman2015stochastic} for the Banach space case.
For what concerns maximum principle for delay equations, after the introduction of the anticipated
backward stochastic differential equations (ABSDEs) in the paper \cite{PengYang}, a wide literature have developed.
We mention, among others, \cite{ChenWuAutomatica2010}
where a problem with pointwise delay in the state and in the control is studied, \cite{oksendal2011optimal}
where a controlled state equation driven by a Brownian motion and by a Poisson random measure is taken into
account, \cite{ChenWuYu2012} where the linear quadratic
case is considered.

\noindent In \cite{MengShenAMO2015} the authors study the stochastic maximum principle for delay evolution equations: the state equation is for some aspects
more general than ours, on the contrary, as in the whole recent literature cited above, the final cost does not depend on the past: in
view of applications this is a strong restriction. 

In the present paper we study the stochastic maximum principle for stochastic control problems where the state equation
is an evolution with delay in the state and where in the cost functional associated 
we allow dependence on the past trajectory also in the final cost; we formulate the maximum principle by means of an adjoint equation
which turns out to be an ABSDE of a new form.

As we just said the recent literature based on ABSDEs does not take into account the case when the final cost
depends on the past trajectory of the state. As far as we know, such a general cost is studied only in \cite{hu1996maximum} for
finite dimensional systems, where the adjoint equation is solved directly by the authors, 
without using the theory of 
ABSDEs. Towards \cite{hu1996maximum}, the novelty of the present paper is that we are able to treat infinite dimensional stochastic control problems.

We also mention that stochastic control problems with delay in the state can also be treated  by the dynamic programming principle and the 
solution of the related Hamilton-Jacobi-Bellmann equation, as it can be done for problems without delay.
For stochastic control problems with delay this approach is carried out in \cite{fuhrman2010stochastic} in
the finite dimensional case and in \cite{zhou2012existence} in the infinite dimensional case.
Both in these two cases the problem is not reformulated in a product space given by the present times the past trajectory,
but it is directly addressed, as we propose here formulating the stochastic maximum principle by means of an ABSDE.\\ 
Comparing the two approaches, we notice that by the stochastic maximum principle we are able to treat state equations 
(see \eqref{eq:state:intro} below) more general than the state equations 
 treated in \cite{fuhrman2010stochastic} and \cite{zhou2012existence}: in the aforementioned papers the state equation has
 to satisfy the structure condition, meaning that the control affects the system only through the noise.
Nevertheless, by the Pontryagin stochastic maximum principle we can only give necessary conditions for optimality for a 
 state equation like (\ref{eq:state:intro}).
On the other side in \cite{fuhrman2010stochastic}
 and in \cite{zhou2012existence}, the authors arrive at more general results, being able to characterize the 
 value functions and the feedback law of the optimal controls.

Before going into details of the matter of the paper, let us present a concrete case we can deal with.
Consider a controlled stochastic heat equation in one dimension,
with Neumann boundary conditions, and with delay in the state:
\begin{equation}\label{heat-eq-contr-intro}
 \left\{
  \begin{array}{l}
  \dis
\frac{ \partial y}{\partial t}(t,\xi)= \Delta y(t,\xi)+\dis \int_{-d}^0 \tilde f(t,y(t+\theta, \xi),u(t,\xi))\mu_f(d\theta) dt 
+\dis\tilde  g(t,\xi)\dot W(t,\xi), \quad t\in [0,T],\;
\xi\in [0,1],
\\\dis
y(\theta,\xi)=x(\theta,\xi),\quad \theta \in [-d,0]
\\\dis
 \dfrac{\partial}{\partial\xi}y(t,\xi)=0, \quad \xi=0,1, \; t\in [0,T].
\end{array}
\right.
\end{equation}
Here $\dot{W}(t,\xi)$ is a space time white noise and $u:\Omega \times [0,T]\times[0,1]\rightarrow \R$ is a control
process such that $u\in L^2(\Omega \times [0,T],L^2([0,1]))$. The maximum delay is given by $d >0$ and $\mu_f$ is a
regular measure on the interval $[-d,0]$. The dependence on the past is given in the drift by
\[
  \int_{-d}^0 \tilde f(t,y(t+\theta,\xi),u(t,\xi))\mu_f(d\theta) , 
\]
so it is of integral type, but with respect to a general regular measure $\mu_f$.
The objective is to minimize a cost functional of the form
\begin{equation}\label{eq:cost:heat:intro}
J(u(\cdot)) = \E \int_0^T \int_{-d}^0 \int_0^1 \tilde l(t,\xi,y(t+\theta), u(t,\xi))d\xi\mu_l(d\theta)dt 
+ \E \int_{T-d}^T \int_0^1 \tilde h(y( \theta,\xi))d\xi\mu_h(d\theta),
\end{equation}
where  $\mu_l$ and $\mu_h$ are regular measures on $[-d,0]$ and $[T-d,T]$, respectively.\\
Equation (\ref{heat-eq-contr-intro}) can be reformulated as an evolution equation in the Hilbert space $H=L^2([0,1])$,
see Section \ref{sez-appl} for more details, and it will be of the form of the following abstract evolution equation we study in the paper.

Namely in a real and separable Hilbert space $H$ we consider the following controlled state equation:
\begin{equation}\label{eq:state:intro}
\begin{system}
dX(t) = \left[ AX(t) + \dis\int_{-d}^0 f(t,X(t+\theta),u(t))\mu_f(d\theta) \right] dt 
+ g(t)dW(t), \\
X_0(\theta) = x(\theta), \qquad \theta \in [-d, 0].
\end{system}
\end{equation}
We assume that $A$ is the generator of a strongly continuous semigroup in $H$, $W$ is a cylindrical Wiener process in another real
and separable Hilbert space $K$ and $f:[0,T]\times H\times U\rightarrow H$ is a Lipschitz
continuous map with respect to the state $X$, uniformly
with respect to the control $u$. 
We refer to Section \ref{sez-appl} for concrete examples of stochastic partial differential equations we can treat.

The control problem associated to (\ref{eq:state:intro}) is to minimize the cost functional
\begin{equation}\label{eq:cost:intro}
J(u(\cdot)) = \E \int_0^T \int_{-d}^0 l(t,X(t+\theta), u(t))\mu_l(d\theta)dt + \E \int_{T-d}^T h(X( \theta))\mu_h(d\theta).
\end{equation}
Both in the state equation and in the cost functional the dependence on the past trajectory is forced to be of
integral type with respect to a general regular measure. The novelty of introducing the dependence on the past trajectory in 
the final cost leads to an additional term in the adjoint equation, namely the adjoint equation, in its mild formulation, is given by
\begin{equation}\label{ABSDEtrascinataMild:intro}
\begin{split}
&p(t) = \int_{t \vee (T-d)}^T\!\! S\left(s-t\right) ' D_x h(X(s)) \mu_h(ds) \\
&\qquad + \int_t^T S(s-t)' \, \E^{\Fcal_s}\!\int_{-d}^0 D_xf(s,X(s),u(s-\theta))'p(s-\theta)\mu_f(d\theta)\, ds \\ 
&\qquad + \int_t^T S(s-t)'\, \E^{\Fcal_s}\!\int_{-d}^0 D_xl(s,X(s+\theta),u(s))\mu_l(d\theta)ds - \int_t^T S(t-s)'q(s)\, dW(s),\\
&p(T-\theta)=0, \;q(T-\theta)=0,\quad \text{ a.e. }\theta \in [-d,0),\, \P-\text{a.s.}.
\end{split}
\end{equation}
If we consider a general regular measure $\mu_h$, not necessarily absolute continuous with respect to the Lebesgue measure, 
the differential form of equation \eqref{ABSDEtrascinataMild:intro}
does not make sense, since the term 
\[
 \frac{d}{dt} \left[
\int_{t \vee (T-d)}^T\!\! S\left(s-t\right)' D_x h(X(s)) \mu_h(ds) \right]
 \]
is not well defined, while the integral form (\ref{ABSDEtrascinataMild:intro}) makes sense: this allows to give the notion of mild solution 
also in this case. In order to be able to work with differentials, 
we will consider an ABSDE where $\mu_h$ is approximated by a sequence of regular measures $(\mu_h^n)_{n\geq 1}$ absolute continuous with 
respect to the Lebesgue measure on $[T-d,T]$, so that once we have also applied the resolvent
operator of $A$, the differential form of this approximating ABSDE makes sense.
Having in hands these tools, we are able to state necessary conditions for the optimality: if
$(\bar{X},\bar{u})$ is an optimal pair for the control problem, given a pair of processes
$$(p,q) \in L^2_\mathcal{F}(\Omega\times [0,T+d];H)\times L^2_\mathcal{F}(\Omega\times [0,T+d];\mathcal{L}_2(K;H))$$ which are 
solution  to the ABSDE \eqref{ABSDEtrascinataMild:intro}, the following variational inequality holds
\begin{equation*}
\<\frac{\partial}{\partial u}\mathcal{H}(t,\bar{X}_t,\bar{u}(t), p(t)), w - \bar{u}(t)\> \geq 0.
\end{equation*}
Here $\mathcal{H}:[0,T]\times C([-d,0],H) \times U \times H \to \R$ is the so called Hamiltonian function and it is given by
\begin{equation}\label{eq:hamiltonian:intro}
\mathcal{H}(t,x,u,p) = \<F(t,x,u),p\>_H
+  L(t,x,u), 
\end{equation}
where $F: C([-d,0],H) \times U\rightarrow H$ and $L:[0,T]\times C([-d,0],H) \times U\rightarrow H$ are defined by
\[
 F(x,u):=\int_{-d}^0 f(x(\theta),u)\mu_f(d\theta) ,\;\;
L(t,x,u):= \int_{-d}^0 l(t,x(\theta),u)\mu_l(d\theta).
 \]

The paper is organized as follows. In Section 2 we present the problem and state our assumptions. Moreover we give some examples of models
we can treat with our techniques. Section 3 is devoted to the study of the state equation and of the related first variation process. 
In Section 4 we introduce the new form of ABSDE which is the essential tool we need to formulate the stochastic maximum principle, and finally
in Section 5 we state and prove the stochastic maximum principle.

\section{Assumptions and preliminaries}\label{sec:AssPrel}

Throughout the paper, we denote by $H$ and $K$ two real and  separable Hilbert spaces with inner product $\<\cdot,\cdot\>_H$ and 
$\<\cdot,\cdot\>_K$ respectively (if no confusion is possible we will use
$\<\cdot,\cdot\>$). We denote by $\mathcal{L}(K; H)$  and $\mathcal{L}_2(K;H)$ the space of linear bounded operators and the Hilbert space of Hilbert-Schmidt operators from $K$ to $H$, respectively.

\noindent Let $(\Omega, \Fcal, \P)$ be a complete probability space
and let $W(t)$ be a cylindrical Wiener process with values in $K$. We endow $(\Omega, \Fcal, \P)$ with the natural filtration $(\mathcal{F}_t)_{t\geq 0}$ generated by $W$ and
augmented in the usual way with the family of $\mP$-null sets of $\mathcal{F}$.
For any $p \geq 1$ and $T>0$ we define
\begin{itemize}
\item $L^p_{\mathcal{F}}(\Omega\times [0,T]; H)$, the set of all $(\mathcal{F}_t)_{t\geq 0}$-progressively measurable processes with values in $H$ such that
\[ \|X\|_{L_{\mathcal{F}}^p(\Omega\times[0,T];H)} = \left( \E \int_0^T |X(t)|_{H}^p dt\right)^{1/p} < \infty; \]
\item $L^p_{\mathcal{F}}(\Omega; C([0,T];H))$, the set of all $(\mathcal{F}_t)_{t\geq 0}$-progressively measurable processes with values in $H$ such that
\[ \|X\|_{L^p_{\mathcal{F}}(\Omega; C([0,T];H))} = \Big( \E \sup_{t \in [0,T]} |X(t)|_{H}^p \Big)^{1/p} < \infty. \]
\end{itemize}

\subsection{Formulation of the control problem}

Let $U$ be another separable Hilbert space and we denote by $U_c$ a convex non empty subspace of $U$.
By admissible control we mean a $(\Fcal_t)_{t\geq 0}$-progressively measurable process with values in $U_c$ such that
\begin{equation}\label{eq:admissible_control}
\E \int_0^T |u(t)|_U^2 dt < \infty.
\end{equation}
We denote by $\calu_{ad}$ the space of admissible controls.
The present paper is concerned with the study of the following controlled evolution equation with delay in $H$

\begin{equation}\label{eq:state}
\begin{system}
dX(t) = \left[ AX(t) +\dis \int_{-d}^0 f(t,X_t(\theta),u(t))\mu_f(d\theta) \right] dt +
g(t)dW(t), \\
X_0 = x(\theta), \qquad \theta \in [-d, 0]
\end{system}
\end{equation}
where $\mu_f$ is a regular measure on $[-d,0]$ with finite total variation, see e.g. \cite{eidelman2004functional}, paragraph 9.9. Moreover $f: [0,T]\times H \times U \to H$, $g: [0,T] \to \Lcal(K;H)$ and $x(\cdot)$ is a given path. 
By $X_t$ we mean the
past trajectory up to time $t-d$, namely
\begin{equation}\label{defdiX_t}
    X_t (\theta)=X(t +\theta),\qquad
\theta \in[-d,0].
\end{equation}
Let us denote by $E:= C([-d,0],H)$ the space of continuous functions from $[-d,0]$ to $H$, endowed with the supremum norm.
It turns out that there exists a continuous solution to \eqref{eq:state}, so we can define
an $E$-valued process $X=(X_{t })_{t \in[0,T]}$ by formula \eqref{defdiX_t}.	
 
Associated to the controlled state equation (\ref{eq:state}) we consider a cost functional of the following form:
\begin{equation}\label{eq:cost}
J(u(\cdot)) = \E \int_0^T \int_{-d}^0 l(t,X_t(\theta), u(t))\mu_l(d\theta)dt + \E \int_{T-d}^T h(X( \theta))\mu_h(d\theta),
\end{equation}
where we stress the fact that the final cost depends on the history of the process. 
Here $l: [0,T] \times H \times U \to \R$, $h: H \to \R$ and $\mu_l,\mu_h$  are real valued regular measures.

The goal is to minimize the cost functional overall admissible controls. We say that a control $\bar{u}$ is optimal if 
\begin{equation}
J(\bar{u}) = \inf_{u(\cdot) \in \mathcal{U}} J(u(\cdot)).
\end{equation}

\begin{example}
Here we give some examples of measures $\mu_f$ and $\mu_l$ we can treat.
\begin{itemize}
\item Linear combinations of Dirac functions,
that is there exist $k_1,...,k_n\in[-d,0]$
such that for every $A \in \calb([-d,0])$ 
\[
 \mu(A):=\sum_{i=1}^n\alpha_i\delta_{k_i}(A), \; k_i\in[-d,0],\, i=1,...,n.
\]
\item Discrete measures with infinite support, as an example we mention the probability measure on $[-d,0]$
given by
\[
 \mu(A):=\sum_{i=1}^\infty 2^{-i}\delta_{k_i}(A),\quad k_i=\frac{-d}{i},\, i\geq 1.
\]
\item Measures which are absolute continuous with respect to the Lebesgue measure on $[-d,0]$, that is there exists 
$w\in L^1([-d,0])$ such that
\[
 \mu(A)=\int_{-d}^0 w(\theta)d\theta
 \]
\end{itemize}
If in the previous points we replace $[-d,0]$ with $[T-d,T]$ we obtain examples of measures $\mu_h$ we can treat.
 \end{example}
 
 We notice that if the measure $\mu_h$ associated to the final cost is the sum of a measure which is absolutely continuous
with respect to the Lebesgue measure in $[T-d,T]$ and of the dirac measure at $T$, that is
\[
 \mu_h(A)=\int_Aw(\theta)d\theta+\delta_{T}(A),
\]
then the cost (\ref{eq:cost}) can be rewritten as
\begin{align*}
J(u(\cdot))& = \E \int_0^T \int_{-d}^0 l(t,X_t(\theta), u(t))\mu_l(d\theta)dt + \E \int_{T-d}^T h(X( \theta))w(\theta)d\theta+\E h(X(T))\\ \nonumber
& = \E \int_0^T \left[\int_{-d}^0 l(t,X_t(\theta), u(t))\mu_l(d\theta)+1_{[T-d,T]}(t)h(X(t))w(t)\right]dt+\E h(X(T)).
\end{align*}
So in this case the cost  can be transformed into a new running cost and a ``standard'' final cost not depending on the past.
The adjoint equation turns out to be a ``standard'' ABSDE 
\begin{equation*}
\begin{split}
&p(t) = \int_{t\vee (T-d)}^T  D_x h(X(s)) w(s)ds \\
&\qquad + \int_t^T S(s-t)' \, \E^{\Fcal_s}\!\int_{-d}^0 D_xf(s-\theta,X(s),u(s-\theta))'p(s-\theta)\mu_f(d\theta)\, ds \\ 
&\qquad + \int_t^T S(s-t)'\, \E^{\Fcal_s}\!\int_{-d}^0 D_xl(s, X(s+\theta),u(s))\mu_l(d\theta)ds - \int_t^T S(t-s)'q(s)\, dW(s),\\
&p(T)= D_x h(X(T)),\;p(T-\theta)=0, \;q(T-\theta)=0,\quad \text{ a.e. } \theta \in [-d,0),\, \P-\text{a.s. }
\end{split}
\end{equation*}
This equation can be also written in differential form as
\begin{equation*}
\begin{split}
&-dp(t) =A'p(t)dt+1_{(T-d,T)}(t) D_x h(X(t)) w(t)dt+  \E^{\Fcal_t}\!\int_{-d}^0 D_xl(X(t+\theta),u(t))\mu_l(d\theta)dt \\
&\qquad \qquad + \E^{\Fcal_t}\!\int_{-d}^0 D_xf(X(t),u(t-\theta))'p(t-\theta)\mu_f(d\theta)\, dt-q(t)\, dW(t), \\ 
&p(T)= D_x h(X(T)),\;p(T-\theta)=0, \;q(T-\theta)=0,\quad \text{ a.e. } \theta \in [-d,0),\, \P\text{ a.s. }
\end{split}
\end{equation*}
This case is a special and simpler case towards the very general final cost we can treat, and for which we are not able
to write directly the adjoint equation in differential form.

\medskip

\begin{remark}\label{remark:L^2}
 We notice that with the approach we present in this paper we can treat costs which are well defined in the space of continuous
 functions $E=C([-d,0],H)$ but not in the space of square integrable functions $L^2([-d,0],H)$, this is the case if the final cost
 is defined by means of a measure $\mu_h$ given e.g. by
 \[
  \mu_h(A):=\sum_{i=1}^\infty 2^{-i}\delta_{k_i}(A),\quad k_i=T-\frac{d}{i},\, i\geq 1,\;A \in \calb([T-d,T]).
 \]
Therefore, the direct approach we propose allows to treat more general functions than the one based on the reformulation of the problem
in the product space $H\times L^2([-d,0],H)$,
which follows the lines of what is done for finite dimensional control problems, see e.g. \cite{da1992stochastic} and references therein.
\end{remark}

\medskip

We now state our assumptions on the coefficients of equation \eqref{eq:state}. In the following, given three real and separable
Hilbert spaces $H_1$, $H_2$ and $H_3$, we denote by $\calg^1(H_1,H_2)$ the class of mappings $F:H_1\rightarrow H_2$ such that
$F$ is continuous, G\^ateaux differentiable on $H_1$ and $DF: H_1\rightarrow \Lcal(H_1,H_2)$ is strongly continuous.
If $H_2=\R$, we write $\calg^1(H_1)$ for $\calg^1(H_1,\R)$. We denote by
$\calg^{0,1}([0,T]\times H_1,H_2)$ the class of continuous functions, which are  G\^ateaux differentiable with respect to $x\in H_1$,
and such that $D_xF:[0,T]\times H_1\rightarrow H_2$ is strongly continuous. We also use the immediate generalization to the class
$\calg^{0,1,1}([0,T]\times H_1\times H_3,H_2)$
We refer to \cite{fuhrman2002nonlinear}, Section 2.2, for more details on the classes $\calg^1$ and $\calg^{0,1}$.

\begin{hypothesis}\label{ipotesi}
On the coefficients of the controlled state equation we assume that
\begin{itemize}
\item[(H0)] the measures $\mu_f$ and $\mu_l$ are finite regular measures on $[-d,0]$ and $\mu_h$ is a finite regular measure
on $[T-d,T]$;
\item[(H1)] the operator $A: D(A) \subset H \to H$ is the infinitesimal generator of a $C_0$-semigroup $S(t)$ in $H$ and there
exist $M>0,\, \omega >0$ such that
\[
 \vert S(t)\vert_{\mathcal{L}(H)}\leq M e^{\omega t};
\]
\item[(H2)]  the map $f: [0,T]\times H \times U \to H$ is measurable, continuous with respect to $x$ and $u$,
and it is Lipschitz continuous with respect to $x$, uniformly with respect to $t\in[0,T]$ and to $u\in U_c$:
$\forall x,\,y\in H$
there exists $C_1>0$ such that
\[
 \vert f(t,x,u)-f(t,y,u)\vert \leq C_1\vert x-y\vert, \quad t\in[0,T],\,u\in U_c;
\]
and moreover for all $u\in U_c$, 
$\forall x\in H$ we assume that
\[
 \vert f(t,x,u)\vert \leq C_1\left(1+\vert x\vert+\vert u\vert\right), \qquad \; t\in[0,T],\,u\in U_c;
\]
\item[(H3)] the map $f$ is G\^ateaux differentiable in $x$ and $u$, moreover $f\in\calg^{0,1,1}([0,T]\times H\times U_c, H)$,
and due to the Lipschitz property of $f$ with respect to $x$ we get that
$D_xf$ is bounded, 
\[
\|D_xf(t,x,u)\|_{\Lcal(H)} \leq C_1,  \quad \; \forall \,t\in[0,T],\,x\in H,\,u\in U_c,
\]
moreover we assume that there exists a constant $C_2>0$ such that
\[\|D_uf(t,x,u)\|_{\Lcal(U;H)} \leq C_2, \quad \; \forall \,t\in[0,T],\,x\in H,\,u\in U_c;
\]
\item[(H4)] the map $g:[0,T]\rightarrow \mathcal{L}(K,H)$ is such that $\forall\, k\in K$
the map $gk:[0,T] \rightarrow H$ is measurable; for every
$s>0,\,t\in[0,T]$ $e^{sA}g(t)\in \Lcal_2(K,H)$ and
the following estimate holds true
 \[
  \vert e^{sA}g(t)  \vert_{ \Lcal_2(K,H)}\leq C s^{-\gamma},\\ \nonumber
  \]
for some constant $C>0$ and $0\leq\gamma< \dfrac{1}{2}.$
\item[(H5)] the map $l: [0,T] \times H \times U \to \R$ belongs to $\calg^{0,1,1}([0,T]\times H\times U_c)$ and there exists $j>0$ and
a constant $C_3>0$ such that
\[
|D_xl(t,x,u)|_H+|D_ul(t,x,u)|_{U} \leq C_3(1 + |x|^j_H + |u|_{U});\]
\item[(H6)] the map $h: H \to \R $ belongs to $\calg^1(H)$
and there exists $k>0$ and
a constant $C_4>0$ such that
\[ |D_x h(x)|_H \leq C_4(1 + |x|^k_H) \]
\end{itemize}
\end{hypothesis}

In the next Section 
 we collect
some results on existence and
uniqueness of a solution to equation (\ref{eq:state}) and on its
regular dependence on the initial condition.

Now we notice that the drift of the state equation (\ref{eq:state}) can be also rewritten
as a linear functional in $E$, and this linear functional enjoys some regularity inherited by Hypothesis
\ref{ipotesi} on the coefficients.
\begin{remark}\label{r:abstract} 
Associated to the coefficient $f$ of the state equation \eqref{eq:state}, we can define a map
$F:[0,T]\times E \times U \to H$ defined by, $\forall\,t\in[0,T],\, x\in E,\, u\in U $,
\begin{equation}\label{FG}
F(t,x,u) := \int_{-d}^0 f(t,x(\theta),u)\mu_f(d\theta).
\end{equation}
The state equation can be rewritten as 
\begin{equation}\label{eq.state:E}
\begin{system}
dX(t) = \left[ AX(t) + F(t,X_t,u(t))\right] dt + g(t)dW(t), \\
X(0) = x(\theta), \qquad \theta \in [-d, 0]
\end{system}
\end{equation}
The same can be done for the cost functional: we can define two maps
$L:[0,T]\times E \times U \to \R$ and $H: E  \to \R$ defined by, $\forall\,t\in[0,T],\, x\in E,\, u\in U $,
\begin{equation}\label{LH}
L(t,x,u) := \int_{-d}^0 l(t,x(\theta),u)\mu_l(d\theta), \qquad H(x) := \int_{T-d}^T h(x(\theta))\mu_h(d\theta)
= \int_{-d}^0 h(x(\theta))\mu^T_h(d\theta),
\end{equation}
where in the last passage we have defined $\mu^T_h$ as the measure on $[-d,0]$ given by
\[
 \mu^T_h(A):=\mu_h(A+T),\;\forall A \in \calb([-d,0]).
\]
The cost functional can be rewritten as
\begin{equation}
\label{cost:E}
J(u(\cdot)) = \E \int_0^T L(t,X_t, u(t)) dt+\E H(X_T).
\end{equation}
\end{remark}
In the following Lemma we show that, under Hypothesis \ref{ipotesi} on $f,\,l,\,h$, the maps
$F,\,L,\,H$ defined in (\ref{FG}) and in (\ref{LH}) satisfy continuity and differentiability
property with respect to the trajectory, which belongs to the space $E$.
\begin{lemma}\label{lemma:regFGLH} Let Hypothesis \ref{ipotesi} holds true and let $F$ be defined by \eqref{FG} and $L,\,H$
be defined by \eqref{LH}. Then $F,\, L,\, H$ turn out to be continuous mappings, moreover
$F\in \calg^{0,1,1}([0,T] \times E\times U_c, H)$,
$L\in \calg^{0,1,1}([0,T]\times E\times U_c)$ and $H\in \calg^1(E)$
\end{lemma}
\begin{proof}
Let us briefly show, for the reader's convenience, one of the previous implications. Suppose that $f$ satisfies Hypothesis 1, then the map 
$F(t,x,u) := \dis\int_{-d}^0 f(t,x(\theta),u) \mu_f(d\theta)$ is Gateaux differentiable
as functions of $x$, that is as a function from $E:= C([-d,0],H)$ to $H$. Indeed, if $h \in E$ we have
\begin{equation}
\begin{split}
D_xF(t,x,u)[h]  &:= \lim_{\varepsilon \to 0} \frac{F(t,x + \varepsilon h,u) - F(t,x,u)}{\varepsilon} \\
&= \frac{1}{\varepsilon} \int_{-d}^0 \left[ f(t,x(\theta) + \varepsilon h(\theta),u) - f(t,x(\theta),u) \right]
\mu_f(d\theta) \\
&= \int_{-d}^0 D_xf(t,x(\theta),u)h(\theta)\mu_f(d\theta)=\<D_xf(t,x,u)\mu_f,h \>_{E^\prime,E}
\end{split}
\end{equation}  
where we used the Gateaux differentiability of the map $f$. In particular, $\| D_F(t,x,u)\|_{\mathcal{L}(E,H)} \leq C$. The other claims follow in the same way.
\end{proof}

\subsection{Some controlled equations we can treat}
\label{sez-appl}

We present here some equations with delay that we can treat with our techniques.

\noindent A first class of models include reaction diffusion equations, such as the stochastic heat equation with
Neumann boundary conditions (\ref{heat-eq-contr-intro}), already mentioned in the Introduction.

Here we present a controlled stochastic heat equation in one dimension,
with Dirichlet boundary conditions, and with delay in the state:
\begin{equation}\label{heat-eq-contr}
 \left\{
  \begin{array}{l}
  \dis
\frac{ \partial y}{\partial t}(t,\xi)= \Delta y(t,\xi)
+\dis \int_{-d}^0 \tilde f(\xi,y(t+\theta, \xi),u(t,\xi))\mu_f(d\theta) dt 
+\dis\tilde  g(t,\xi)\dot{W}(t,\xi), \quad t\in [0,T],\;
\xi\in [0,1],
\\\dis
y(\theta,\xi)=x(\theta,\xi),\quad \theta \in [-d,0],
\\\dis
 y(t,\xi)=0, \quad \xi=0,1, \; t\in [0,T].
\end{array}
\right.
\end{equation}
The process $\dot{W}(t,\xi)$ is a space time white noise and $u:\Omega \times [0,T]\times[0,1]\rightarrow \R$ is the control
process such that $u\in L^2(\Omega \times [0,T],L^2([0,1]))$. The maximum delay is given by $d>0$ and $\mu_f$ is a
regular measure on the interval $[-d,0]$.
The cost functional we can study is given by
\begin{equation}\label{eq:cost:heat}
J(u(\cdot)) = \E \int_0^T \int_{-d}^0 \int_0^1 \tilde l(t,\xi,y(t+\theta,\xi), u(t,\xi))d\xi\mu_l(d\theta)dt 
+ \E \int_{T-d}^T \int_0^1 \tilde h(y( \theta,\xi),\xi)d\xi\mu_h(d\theta).
\end{equation}
The measure $\mu_l$ is a regular
measure on $[-d,0]$, and $\mu_h$ is a regular measure on $[T-d,T]$.

\noindent In abstract reformulation equation (\ref{heat-eq-contr}) is an 
evolution equation in the Hilbert space $L^2([0,1])$ and the space of controls $U$ is itself given by $H=L^2([0,1])$:
equation (\ref{heat-eq-contr}) can be rewritten as
\begin{equation}\label{eq:state:heatabstract}
\begin{system}
dX(t) = \left[ AX(t) +\dis \int_{-d}^0 f(t,(X(t+\theta),u(t))\mu_f(d\theta) \right] dt 
+ g(t)dW(t), \\
X_0 (\theta)= x(\theta), \qquad \theta \in [-d, 0],
\end{system}
\end{equation}
where $X(t)(\xi):=y(t,\xi)$, $A$ is the Laplace operator with Dirichlet boundary conditions, 
$f$ is the  Nemytskii operator associated to $\tilde f$ , and $g$ is the multiplicative operator associated to $\tilde g$ by
\[
 (g(t) z)(\xi):=\tilde g (t,\xi) z(\xi), \text{ for a.e. }\xi \in [0,1], \; \forall \,z\in\,L^2([0,1]).
\]
The process $W$ is a cylindrical Wiener process in $L^2([0,1])$. In abstract reformulation the cost functional can be rewritten as
\begin{equation}\label{eq:cost:heatabstract}
J(u(\cdot)) = \E \int_0^T \int_{-d}^0  l(t,X(t+\theta), u(t))\mu_l(d\theta)dt 
+ \E \int_{T-d}^T  h(X(\theta))\mu_h(d\theta)
\end{equation}
where
\[
  l(t,X(t+\theta), u(t))=\int_0^1 \tilde l(t,\xi,X(t+\theta)(\xi), u(t,\xi))d\xi, \text{ and } 
h(X(\theta))= \int_0^1\tilde h(y(\theta,\xi),\xi)d\xi.
\]
We now make some assumptions on the coefficients and on the cost functional
so that Hypothesis 1 is satisfied.
\begin{hypothesis}\label{heat-ipotesi}
 The functions $\tilde f,$ $\tilde g$, $\tilde l$ and $\tilde h$ are all
measurable and real valued. Moreover
\begin{enumerate}
\item $\tilde f$ is defined on $
[0,1]\times\R\times \R$ and there exists a constant $L>0$ such that
\[
\left|  \tilde{f}\left(\xi,  x,u\right)  -\tilde{f}\left(\xi,
y,u\right)  \right|  \leq L\left|  x-y\right|  ,
\]
for almost all $\xi\in[0,1]$, for all $x,$ $y\in\mathbb{R}$ and $u\in U$. Moreover 
for almost all $\xi\in[0,1]$, $\forall u \in \R$, $\tilde{f}\left(\xi,  \cdot,u\right)  \in
C^{1}_b\left(  \mathbb{R}\right)  $ and $\forall x \in \R$, $\tilde{f}\left(\xi, x,\cdot\right)  \in
C^{1}_b\left(  \mathbb{R}\right)  $;

\item $\tilde g$ is defined on $[0,1]\times \left[  0,T\right]  $ and there exist a constants  $K>0$ such
that for almost all $\xi\in[0,1]$ and for $t\in[0,T]$
\[
\left|  \tilde g\left(\xi, t\right) \right|  \leq K;
\]

\item $\tilde l:\left[  0,T\right]  \times\left[  0,1\right]  \times\mathbb{R
}\times \R  \rightarrow\mathbb{R}$ is measurable and for
a.a. $t\in\left[  0,T\right]  ,$ $\xi\in\left[  0,1\right]  ,$ the map
$\tilde l\left(  \tau,\xi,\cdot,\cdot\right)  :\mathbb{R}\times\R  \rightarrow\mathbb{R}$ is continuous. There exists  $c_{1}$
square integrable on $\left[  0,1\right]  $ such that
\[
\left|\tilde  l\left(  t,\xi,x,u\right)  -\tilde l\left(  t,\xi,y,u\right)  \right|
\leq c_{1}\left(  \xi\right)  \left|  x-y\right|  ,
\]
for $\xi\in\left[  0,1\right]  $, $x,y\in\mathbb{R}$, $t\in\left[
0,T\right]  $, and $u\in\R $. Moreover $\left|
\tilde l\left(  t,\xi,x,u\right)  \right|  \leq c_{2}\left(  \xi\right)  ,$ with
$c_{2}$ integrable on $\left[  0,1\right]  $. In addition, for a.a. $\xi\in [0,1]$ and $t\in[0,T]$, and
for $x,u\in \R$
$\tilde l(t,\xi,\cdot, u)\in C^1_b(\R)$
and $\tilde l(t,\xi,x,\cdot)\in C^1_b(\R)$.

\item $\tilde h$ is defined on $\R\times \left[  0,1\right]  \times\mathbb{R}$ and $\tilde h\left(
\cdot,\xi\right)  $ is uniformly continuous, uniformly with respect to $\xi
\in\left[  0,1\right]  $; moreover $\left|  \tilde h\left(x,  \xi\right)  \right|
\leq c_{3}\left(  \xi\right)  ,$ with $c_{3}$ integrable on $\left[
0,1\right]  $;

\item $x_{0}\in L^{2}\left(  \left[  0,1\right]  \right)  $.
\end{enumerate}
\end{hypothesis}
It is immediate to see that if Hypothesis \ref{heat-ipotesi} holds true for the coefficients
$\tilde f$, $\tilde g$, $\tilde l$ and $\tilde h$, then $f$, $g$, $l$ and $h$ satisfy Hypothesis \ref{ipotesi}.

\bigskip

Another class of systems we can treat is given by controlled stochastic wave equations:
we consider, for $0\leq t\leq T$ and $\xi\in\left[
0,1\right]  $, the following wave equation:
\begin{equation}
\left\{
\begin{array}
[c]{l}%
\frac{\partial^{2}}{\partial t^{2}}y\left(t,\xi\right)  =\frac
{\partial^{2}}{\partial\xi^{2}}y\left(t,\xi\right)+\dis\int_{-d}^0\tilde f\left(\xi,y(t+\theta,\xi),u\left(t,\xi\right)\right)
\mu_f(d\theta)
+\dot{W}\left(t,\xi\right) \\
y\left(t,0\right)  =y\left(t,1\right)  =0,\\
y\left(0,\xi\right)  =x_{0}\left(  \xi\right)  ,\\
\frac{\partial y}{\partial t}\left(t, \xi\right)\mid_{t=0}  =x_{1}\left(  \xi\right),
\end{array}
\right.  \label{waveequation}%
\end{equation}
where $\dot{W}_ t\left(\xi\right)  $ is a space-time white noise on $\left[
0,T\right]  \times\left[  0,1\right]  $ and $u $ is the control process in $L^2(  \Omega\times [0,T],$ $L^{2}\left(  0,1\right))$. The cost functional we are interested in is the following
\begin{equation}
\begin{split}
J\left(  x_{0},x_{1},u\right)  &=\E\int_{0}^{T}\int_{-d}^0\int_{0}^{1}
\tilde{l}\left(  t,\xi,y\left(t+\theta,\xi\right),u\left(t,\xi\right)  \right)
d\xi \mu_l(d\theta)dt\\
&+\E\int_{T-d}^T\int_{0}^{1}\tilde h\left(  \xi,y\left(\theta,\xi\right)  \right)\mu_h(d\theta)
d\xi.\label{costo:wave}
\end{split}
\end{equation}
The measures $\mu_f$ and $\mu_l$ are regular measures on $[-d,0]$, and $\mu_h$ is a regular measure on $[T-d,T]$. On the drift $\tilde f$
and on the costs $\tilde l$ and $\tilde h$ we make the following assumptions:
\begin{hypothesis}
\label{ip costo wave} The functions $\tilde f,$ $\tilde g$, $\tilde l$ and $\tilde h$ are all
measurable and real valued. Moreover
\begin{enumerate}
\item $\tilde f$ is defined on $\left[  0,1\right]
\times\mathbb{R} \times \R $ and there exists a constant $C>0$ such that,
for a.a.  $\xi\in\left[  0,1\right]$, $\forall x,y,u\in \R$ 
\[
\left|  \tilde f\left( \xi,x,u\right)-\tilde f\left(\xi,y,u\right)  \right|
\leq C\left|  x-y\right|  .
\]
Moreover for almost all $\xi\in[0,1]$, for all $x,y\in\R$, $f\left(  \xi,\cdot,u\right)  \in
C_b^{1}\left(  \mathbb{R}\right)  $ and $f\left(  \xi,x,\cdot\right)  \in
C_b^{1}\left(  \mathbb{R}\right)  $.

\item $\tilde l:\left[  0,T\right]  \times\left[  0,1\right]  \times\mathbb{R
}\times \R  \rightarrow\mathbb{R}$ is such that for
a.a. $t\in\left[  0,T\right]  ,$ $\xi\in\left[  0,1\right]  ,$ the map
$\tilde l\left(  \tau,\xi,\cdot,\cdot\right)  :\mathbb{R}\times\R  \rightarrow\mathbb{R}$ is continuous. There exists  $c_{1}$
square integrable on $\left[  0,1\right]  $ such that
\[
\left|\tilde  l\left(  t,\xi,x,u\right)  -\tilde l\left(  t,\xi,y,u\right)  \right|
\leq c_{1}\left(  \xi\right)  \left|  x-y\right|  ,
\]
for $\xi\in\left[  0,1\right]  $, $x,y\in\mathbb{R}$, $t\in\left[
0,T\right]  $ and $u\in\R $. There exists $c_{2}$ integrable on $\left[  0,1\right]  $ such that  $\left|
\tilde l\left(  t,\xi,x,u\right)  \right|  \leq c_{2}\left(  \xi\right) $. Moreover for a.a. $\xi\in [0,1]$ and $t\in[0,T]$, and
for $x,u\in \R$
$\tilde l(t,\xi,\cdot, u)\in C^1_b(\R)$
and $\tilde l(t,\xi,x,\cdot)\in C^1_b(\R)$.

\item $\tilde h$ is defined on $\R\times \left[  0,1\right]  \times\mathbb{R}$ and $\tilde h\left(
\cdot,\xi\right)  $ is uniformly continuous, uniformly with respect to $\xi
\in\left[  0,1\right]  $; moreover $\left|  \tilde h\left(x,  \xi\right)  \right|
\leq c_{3}\left(  \xi\right)  ,$ with $c_{3}$ integrable on $\left[
0,1\right]  $ and for almost all $\xi\in[0,1]$ $\tilde h(\cdot,\xi)\in C^1_b(\R)$;

\item $x_{0}$, $x_{1}\in L^{2}\left(  \left[  0,1\right]  \right)  $.

\end{enumerate}

\end{hypothesis}

We want to write equation (\ref{waveequation}) in an abstract form. We denote by $\Lambda$ the Laplace operator with Drichlet
boundatry conditions and we introduce
the Hilbert space
\[
H=L^{2}\left(  \left[  0,1\right]  \right)  \oplus\mathcal{D}\left(
\Lambda^{-\frac{1}{2}}\right)  =L^{2}\left(  \left[  0,1\right]  \right)
\oplus H^{-1}\left(  \left[  0,1\right]  \right)  .
\]On $H$ we define the operator $A$ by
\[
\mathcal{D}\left(  A\right)  =H_{0}^{1}\left(  \left[  0,1\right]  \right)
\oplus L^{2}\left(  \left[  0,1\right]  \right)  ,\text{ \ \ \ \ }A\left(
\begin{array}
[c]{c}%
y\\
z
\end{array}
\right)  =\left(
\begin{array}
[c]{cc}%
0 & I\\
-\Lambda & 0
\end{array}
\right)  \left(
\begin{array}
[c]{c}%
y\\
z
\end{array}
\right)  ,\text{ \ for every }\left(
\begin{array}
[c]{c}%
y\\
z
\end{array}
\right)  \in\mathcal{D}\left(  A\right)  .
\]
We also set $f:H\times L^2([0,1])\rightarrow L^2([0,1])$ given by
\[
 f\left(\left(\begin{array}
[c]{c}%
y\\
z
\end{array}\right),u \right)\left(\xi\right):=\tilde f(\xi, y(\xi), u(\xi)), \quad  \text{for all } \left(\begin{array}
[c]{c}%
y\\
z
\end{array} \right)\in H, \, \xi \in[0,1],
\]
and 
 $g:L^{2}\left(  \left[  0,1\right]  \right)  \longrightarrow H$
with $$
gu=\left(
\begin{array}
[c]{c}%
0\\
u
\end{array}
\right)  =\left(
\begin{array}
[c]{c}%
0\\
I
\end{array}
\right)  u
,$$
for all $u\in L^2([0,1])$.

\noindent Equation (\ref{waveequation}) can be rewritten in an abstract way as an
equation in $H$ of the following form:%
\begin{equation}
\left\{
\begin{array}
[c]{l}%
dX(t) =AX(t)dt+g\int_{-d}^0f(\left( X(t+\theta), u(t)\right)dt+gdW_t  ,\text{ \ \ \ }t\in\left[
0,T\right] \\
X_0 =x.
\end{array}
\right.  \label{waveeqabstract}%
\end{equation}
The cost functional (\ref{costo:wave})
\begin{equation}\label{eq:cost:waveabstract}
J(u(\cdot)) = \E \int_0^T \int_{-d}^0  l(t,X(t+\theta), u(t))\mu_l(d\theta)dt 
+ \E \int_{T-d}^T  h(X(\theta))\mu_h(d\theta)
\end{equation}
where
\[
  l(t,X(t+\theta), u(t))=\int_0^1 \tilde l(t,\xi,X_1(t+\theta,\xi), u(t,\xi))d\xi, \text{ and } 
h(X(\theta))= \int_0^1\tilde h(X_1(\theta,\xi),\xi)d\xi.
\]
where $X_1,X_2$ are the first and the second component of an element $X\in H$, namely
$X=\left( 
\begin{array}{l}
X_1\\
X_2
\end{array}
\right)$.
It is immediate to see that if Hypothesis \ref{ip costo wave} holds true for the coefficients
$\tilde f$, $\tilde g$, $\tilde l$ and $\tilde h$, then $f$, $g$, $l$ and $h$ satisfy Hypothesis \ref{ipotesi}.

We notice that, since in its abstract formulation equation (\ref{waveequation}) reads as (\ref{waveeqabstract}), it satisfies the structure condition,
and so the control problem could be treated by solving the associated Hamilton Jacobi Bellmann equation, as in \cite{zhou2012existence}.
With our techniques we can treat equations more general than equation (\ref{waveequation}), for example  with a diffusion term $\tilde g :[0,1]
\to\R$ not invertible.

\section{Analysis of the state equation}

In this Section we study the state equation 
(\ref{eq:state}) and its behaviour with respect to a convex perturbation of the optimal control. Here and in the following, for the sake of brevity, we consider a drift $f$ and a current cost
$l$ which do not depend on time.
A mild solution to equation (\ref{eq:state}) is 
an $\Fcal_t$-progressively measurable process, satisfying $\P$-a.s., for $t\in[0,T]$ the integral equation
\begin{equation}\label{eq:state.mild}
\begin{split}
X(t) &= S(t)x(0) + \int_0^t S(t-s) \int_{-d}^0 f(X_s(\theta),u(s))\mu_f(d\theta) ds + \int_0^t S(t-s)g(s)dW(s).
\end{split}
\end{equation}
We refer e.g. to \cite{da1992stochastic} for the basic properties of mild solution
of evolution equations.
In the following Proposition we prove existence of a mild solution for equation (\ref{eq:state}), for every admissible control $u$, and we show that this mild solution has smooth dependence with respect to the initial condition.

\begin{proposition}\label{prop-esistenza} Let Hypothesis \ref{ipotesi}, points H0, H1, H2, H4, holds true. Then there exists a unique mild solution $X$
to equation (\ref{eq:state}) and for every $p\geq 1$  there exists a constant $c >0$ such that
\begin{equation}\label{stimafor}
\E\sup_{t\in [-d,T]}|X(t)|^p_H \leq c (1+ |x|)^p_{C([-d,0];H)}.
\end{equation} 
\end{proposition}
\begin{proof}This result is collected in \cite{zhou2012existence}, Theorem 3.2. In \cite{zhou2012existence} no proof is given,
referring to \cite{fuhrman2002nonlinear}, where the case of $H$-valued stochastic evolution equations without delay is considered.

The proof follows by techniques in \cite{fuhrman2002nonlinear} and in \cite{mohammed1998stochastic}, we just give a sketch.
We introduce the map $$\Gamma: L^p_\Fcal(\Omega,C([0,T],E))\mapsto L^p_\Fcal(\Omega,C([0,T],E))$$
defined by, $\forall \theta \in [-d,0]$, 
\begin{equation}\label{Gamma}
 \left(\Gamma (Y)_t\right)(\theta):= \tilde \Gamma(Y)_{t+\theta},
\end{equation}
where for all $s\in [-d,T]$, we define $\tilde \Gamma_s : L^p_\Fcal(\Omega,C([0,T],H))\mapsto L^p_\Fcal(\Omega,C([0,T],H))$ as
\begin{align}\label{tildeGamma}
&\tilde \Gamma(Y)_s=x(s),\qquad s\in [-d,0],\\ \nonumber
& \tilde \Gamma(Y)_s=S(s)x(0) + \int_0^s S(s-r) \int_{-d}^0 f(Y(r+\theta),u(r))\mu_f(d\theta) dr \\ \nonumber 
&\qquad \quad + \int_0^t S(s-r)g(r)dW(r), \quad s\in [0,T].
\end{align}
It turns out that $\Gamma$ is well defined from $L^p_\Fcal(\Omega,C([0,T],E))$ to $L^p_\Fcal(\Omega,C([0,T],E))$
and it is a contraction: to this aim it suffices to notice that
\[
 \E\sup_{t\in [-d,T]}|X(t)|^p=\E\sup_{t\in [0,T]}\|X_t\|_E^p
\]
and to apply arguments in \cite{fuhrman2002nonlinear}, Propositions 3.2 and 3.3.
The crucial point is the fact that $f$ is Lipschitz continuous and the
measure $\mu_f$ has finite total variation.
\end{proof}

Let $ (\bar{X}(\cdot),\bar{u}(\cdot))$ be an optimal pair. For any $w(\cdot) \in \mathcal{U}_{ad}$ we can define the perturbed control  
\begin{equation}\label{u^rho}
u^{\rho}(\cdot) :=  \bar{u}(\cdot) + \rho v(\cdot) 
\end{equation}
 where $v(\cdot) = w(\cdot) - \bar u(\cdot)$ and $0\leq \rho \leq 1$. The perturbed control $u^{\rho}(\cdot)$ is admissible due to the convexity of the set $U_c$ and the corresponding state is denoted by $X^{\rho}$. 

\begin{lemma}\label{lemma:x epsilon}
Under Hypothesis 1, points H0, H1, H2 and H4, the following holds
\begin{equation*}
\E \operatorname{sup}_{t \in [0,T]} |X^{\rho}(t) - \bar{X}(t)|_H^2 \rightarrow 0, \quad  \, \text{as } \,\rho \rightarrow 0.
\end{equation*}
\end{lemma}
\begin{proof}
Let us write the equation satisfied by $X^\rho - \bar{X}$ in mild form
\begin{equation}
 X^{\rho}(t) - \bar{X}(t) = \int_0^t S(t-s) \int_{-d}^0 \left[ f(X^{\rho}(s+\theta),u^{\rho}(s)) - f(\bar{X}(s+\theta),\bar{u}(s)) \right]\mu_f(d\theta) ds. 
\end{equation}
Then we have
\begin{equation*}
\begin{split}
 |X^{\rho}(t) - \bar{X}(t)|_H^2 &\leq C \int_0^t \|S(t-s)\|^2_{\Lcal (H)} \int_{-d}^0 |f(X^{\rho}(s+\theta),u^{\rho}(s)) - f(\bar{X}(s+\theta),\bar{u}(s))|^2_H|\mu_f|(d\theta) ds \\
 &\leq C \int_0^t \|S(t-s)\|^2_{\Lcal (H)} \int_{-d}^0 |f(X^{\rho}(s+\theta),u^{\rho}(s)) - f(\bar{X}(s+\theta),u^{\rho}(s))|^2_H|\mu_f|(d\theta) ds \\
 &+ C \int_0^t \|S(t-s)\|^2_{\Lcal (H)} \int_{-d}^0 |f(\bar{X}(s+\theta),u^{\rho}(s)) - f(\bar{X}(s+\theta),\bar{u}(s))|^2_H|\mu_f|(d\theta) ds. \\
\end{split}
\end{equation*}
Thanks to Hypotheses H1, H2 and the finiteness of the total variation measure $|\mu_f|$, taking the supremum in $t \in [0,T]$ we obtain
\begin{equation}
\begin{split}
 |X^{\rho}(t) - \bar{X}(t)|_H^2 &\leq K \int_0^t \sup_{r \in [0,s]} |X^{\rho}(r) - \bar{X}(r)|_H^2ds +  \nu^{\rho}(t) \\
 &\leq K\int_0^T \sup_{r \in [0,s]} |X^{\rho}(r) - \bar{X}(r)|_H^2ds +  \nu^{\rho}(T)
\end{split}
\end{equation}
where $K$ depends only on $T$ and we denoted by $\nu^{\rho}(t)$ the quantity
\[ \nu^{\rho}(T):= \int_0^T \int_{-d}^0 |f(\bar{X}(s+\theta),u^{\rho}(s)) - f(\bar{X}(s+\theta),\bar{u}(s))|^2_H|\mu_f|(d\theta) ds. \] 
Let us notice that, by H2, by estimates \eqref{stimafor}, and the definition of $u^\rho$ in \eqref{u^rho}, we can apply
the Dominated Convergence Theorem. By the continuity of the map $f$ with respect to $u$, we get $\nu^\rho(t) \to 0$ if $\rho \to 0$.
Now we can take the supremum over the entire interval $[0,T]$ to obtain
\begin{equation*}
\sup_{t \in [0,T]}|X^{\rho}(t) - \bar{X}(t)|_H^2 \leq K \int_0^T \sup_{t \in [0,s]}|X^{\rho}(t) - \bar{X}(t)|_H^2 ds + \nu^{\rho}(T).
\end{equation*}
Using the Gronwall Lemma we get
\begin{equation}
\E\sup_{t \in [0,T]}|X^{\rho}(t) - \bar{X}(t)|_H^2 \leq \tilde{K}\E\left[ \nu^{\rho}(T)\right],
\end{equation}
and letting $\rho \to 0$ we get the required result. 
\end{proof}

Now we can introduce the first variation process $Y(t)$, which satisfies the following equation
\begin{equation}\label{eq:first.variation}
\left\lbrace\begin{array}{l}
\frac{d}{dt}Y(t) = AY(t) + \dis\int_{-d}^0 D_xf(X(t+\theta),u(t))Y(t+\theta)\mu_f(d\theta) + 
\dis\int_{-d}^0D_uf(X(t+\theta),u(t))\mu_f(d\theta) v(t)  \\
Y_0 = 0.   
\end{array}\right.
\end{equation}

This equation is well-posed in a mild sense, as we show in the following
\begin{proposition}
Let Hypothesis 1 be in force. Then equation \eqref{eq:first.variation} admits a unique mild solution solution, i.e. a progressive $H$-valued process $Y \in L^2_{\Fcal}(\Omega \times [0,T];H)$ such that for all $t \in [0,T]$, $\mP$-a.s.
\begin{equation}
Y(t) = \int_0^t S(t-s)\int_{-d}^0 \left[D_xf(X(s+\theta),u(s))Y(s+\theta) + D_uf(X(s+\theta),u(s))v(s) \right]\mu_f(d\theta)ds \\
\end{equation}  
\end{proposition}
\begin{proof}
The proof follows like
the one of Proposition \ref{prop-esistenza}, the equation here is linear and there is no diffusion term.
\end{proof}
Once we have defined the the first variation process $Y$, we can write an expansion up to first order of the perturbed trajectory of the state
\begin{equation}\label{resto:fin:rho}
 X^\rho(t)=\bar{X}(t)+\rho Y(t)+R^\rho(t),\, \quad t\in[0,T].
\end{equation}
The aim is to show that the rest goes to zero in the right topology, i.e.
\[\lim_{\rho\to 0}\frac{1}{\rho^2}\E\sup_{t\in[0,T]}\vert R^\rho(t)\vert^2=0,\]
This is the content of the following
\begin{lemma}
Let Hypothesis 1, points H0, H1, H2, H3, H4,  holds. Then the process $\tilde{X}^{\rho}$ defined as
\begin{equation}\label{resto:fin:rho:frac}
\tilde{X}^{\rho}(t) = \dfrac{X^{\rho}(t) - \bar{X}(t)}{\rho} - Y(t),
\end{equation}
satisfies the following
\begin{equation}\label{lemma:convergence X tilde}
\lim_{\rho \rightarrow 0}\E \sup_{t \in [0,T]}|\tilde{X}^{\rho}(t)|^2 = 0.
\end{equation}
\end{lemma}

\begin{proof}
To lighten the notation, here we adopt the convention introduced in \eqref{defdiX_t}. 
The process $\tilde{X}^{\rho}$ defined in (\ref{resto:fin:rho:frac}) is the solution to the following integral equation
\begin{equation*}
\begin{split}
\tilde{X}^{\rho}(t) &= \int_0^t S(t-s)\frac{1}{\rho} \int_{-d}^0 \bigl[ f( \bar{X}_s(\theta) + \rho Y_s( \theta) + \rho \tilde{X}^{\rho}_s(\theta),\bar{u}(s)+\rho v(s)) - f(\bar{X}_s(\theta),\bar{u}(s)) \bigr]\mu_f(d\theta)ds \\
&- \int_0^t S(t-s) \int_{-d}^0 \bigl[ D_xf(\bar{X}_s(\theta),\bar{u}(s)) Y_s(\theta) + D_u f(\bar{X}_s(\theta),\bar{u}(s)) v(s)\bigr] \mu_f(d\theta) ds 
\end{split}
\end{equation*}
with $\tilde{X}^{\rho}(0) = 0 $ as initial datum. Via standard computations
we obtain
\begin{equation*}
\begin{split}
\tilde{X}^{\rho}(t) &= \int_0^tS(t-s) \int_0^1\int_{-d}^0 D_xf\bigl(\bar{X}_s(\theta) + \lambda\rho( Y_s(\theta) + \tilde{X}^{\rho}_s(\theta)),\bar{u}(s)\bigr)\tilde{X}^{\rho}_s(\theta)\, \mu_f(d\theta) d\lambda ds\\
&+ \int_0^tS(t-s)\int_0^1 \int_{-d}^0 \bigl[ D_xf\bigl(\bar{X}_s(\theta) + \lambda\rho( Y_s(\theta) + \tilde{X}^{\rho}_s(\theta)),\bar{u}(s)\bigr) \\
&\qquad - D_xf(\bar{X}_s(\theta),\bar{u}(s))\bigr]Y_s(\theta)\mu_f(d\theta) d\lambda ds \\
&+ \int_0^tS(t-s)\int_0^1 \int_{-d}^0 \bigl[ D_u f\bigl(\bar{X}_s(\theta) + \lambda\rho( Y_s(\theta) + \tilde{X}^{\rho}_s(\theta)),\bar{u}(s)+ 
\lambda\rho v(s)\bigr)\\
&\qquad  - D_uf(\bar{X}_s(\theta),\bar{u}(s))\bigr]v(s)\, \mu_f(d\theta)d\lambda ds ,
\end{split}
\end{equation*}
so that
\begin{equation*}
\begin{split}
|\tilde{X}^{\rho}(t)|^2_H &\leq K\int_0^t \|S(t-s)\|^2_{\Lcal(H)} \Big|\int_{-d}^0 \tilde{X}^{\rho}_s(\theta) \mu_f(d\theta)\Big|^2_H ds\\
 & + \int_0^T  \|S(T-t)\|^2_{\Lcal(H)}\Big| \int_0^1 \int_{-d}^0 \bigl[ D_xf\bigl(\bar{X}_t(\theta) + \lambda\rho( Y_t(\theta) + \tilde{X}^{\rho}_t(\theta)),\bar{u}(t)\bigr) \\
 &\qquad - D_xf(\bar{X}_t(\theta),\bar{u}(t))\bigr]Y_t(\theta)\mu_f(d\theta) d\lambda \Big|^2_H dt \\
 &+ \int_0^T \|S(T-t)\|^2_{\Lcal(H)}\Big| \int_0^1 \int_{-d}^0 \bigl[ D_u f\bigl(\bar{X}_t(\theta) + 
 \lambda\rho( Y_t(\theta) + \tilde{X}^{\rho}_t(\theta)),\bar{u}(t)+ 
 \lambda\rho v(t)\bigr)\\
 &\qquad  - D_uf(\bar{X}_t(\theta),\bar{u}(t))\bigr]v(t)\, \mu_f(d\theta)d\lambda \Big|_H^2 dt \\ 
&\leq K \int_0^T \Big|\int_{-d}^0 \tilde{X}^{\rho}_s(\theta) \mu_f(d\theta)\Big|^2_H ds + \nu_{\rho}(T),
\end{split}
\end{equation*}
where the constant $K$ only depends on $T$ and $\nu_{\rho}(T)$ is defined by
\begin{equation}
\begin{split}
\nu_{\rho}(T)&=  \int_0^T \|S(T-t)\|^2_{\Lcal(H)}
\Big| \int_0^1 \int_{-d}^0 \bigl[ D_u f\bigl(\bar{X}_t(\theta) + 
 \lambda\rho( Y_t(\theta) + \tilde{X}^{\rho}_t(\theta)),\bar{u}(t)+ 
 \lambda\rho v(t)\bigr)\\
 &\quad - D_uf(\bar{X}_t(\theta),\bar{u}(t))\bigr]v(t)\, \mu_f(d\theta)d\lambda \Big|_H^2 dt .
\end{split}
\end{equation}
By the boundedness of $D_xf(\cdot)$ and $D_uf(\cdot)$, we have that $\nu_\rho(t)\to 0$ as $\rho \to 0$. 
\begin{equation}
\begin{split}
|\tilde{X}^{\rho}(t)|^2_H &\leq K\int_0^T \sup_{r \in [0,s]}|\tilde{X}^{\rho}(r)|^2_H ds  + \nu_{\rho}(T) \\
&\leq K\int_0^T \sup_{r \in [0,s]}|\tilde{X}^{\rho}(r)|^2_H ds  + \nu_{\rho}(T),
\end{split}
\end{equation}
Now we can take the supremum on the left hand side 
\begin{equation}
\E\sup_{t \in [0,T]}|\tilde{X}^{\rho}(t)|^2_H \leq K\E\int_0^T \sup_{r \in [0,s]}|\tilde{X}^{\rho}(r)|^2_H ds  + \E\left[ \nu_{\rho}(T) \right].
\end{equation}
Using the Gronwall inequality and the convergence $\nu_\rho(t)\to 0$, as $\rho \to 0$, we get the result.
\end{proof}

\section{A new form of Anticipated Backward SPDE}

In this section we study ABSDEs with suitable form to be the adjoint equation in problems with delay. The solvability of this class of equations allows to formulate a stochastic maximum principle for infinite dimensional controlled state equations with delay, with final cost functional depending on the history of the process.
 
 Many recent papers, we cite among others  \cite{ChenWuAutomatica2010}, \cite{oksendal2011optimal}, \cite{ChenWuYu2012} and \cite{MengShenAMO2015},
 deal with similar problems, but only in the case of final cost
 not depending on the past of the process. Such a general case is treated in the
 paper \cite{hu1996maximum} for finite dimensional systems, and it is here proved by means of ABSDEs. 

We will consider the following infinite dimensional ABSDE, which in integral form is given by
\begin{equation}\label{ABSDEtrascinataMild}
\begin{split}
p(t) &= \int_{t \vee (T-d)}^T\!\! S\left(s-t\right) ' D_x h(X(s)) \mu_h(ds) \\
&+ \int_t^T S(s-t)' \, \E^{\Fcal_s}\!\int_{-d}^0 D_xf(X(s),u(s-\theta))'p(s-\theta)\mu_f(d\theta)\, ds \\ 
&+ \int_t^T S(s-t)'\, \E^{\Fcal_s}\!\int_{-d}^0 D_xl(X(s+\theta),u(s))\mu_l(d\theta)ds - \int_t^T S(t-s)'q(s)\, dW(s),  \quad t \in [0,T]\\
p(t)&=0,  \qquad \text{for all } t \in ]T,T+d], \P-\text{a.s.},\qquad   q(t)=0 \quad  \text{a.e. } t \in ]T,T+d], \P-\text{a.s.} .
\end{split}
\end{equation}
If we do not make additional assumption on the measure $\mu_h$, the differential form of equation (\ref{ABSDEtrascinataMild})
does not make sense, since the term 
\[
 d _t \left[
\int_{t \vee (T-d)}^T\!\! S\left(s- t\right) ' D_x h(X(s)) \mu_h(ds) \right]
 \]
is not well defined for every $t \in [0,T]$, while the notion of mild solution is meaningful also in this case.
\begin{definition}
We say that \eqref{ABSDEtrascinataMild} admits a mild solution if there is a couple of
processes
$$(p,q) \in L^2_\mathcal{F}(\Omega\times [0,T+d];H)
\times L^2_\mathcal{F}(\Omega\times [0,T+d]; \mathcal{L}_2(K;H))$$
that solves equation \eqref{ABSDEtrascinataMild} for all $t \in [0,T+d]$.
\end{definition}
We will prove existence and uniqueness for such equation by an approximation procedure. We
start by recalling  the following apriori estimates.

\begin{lemma}\label{l:apriori.esitmate.backward}
Let $\gamma \in L^2_{\Fcal}(\Omega \times [0,T];H)$ be a progressively measurable process and $\xi \in L^{2}_{\Fcal_T}(\Omega; H)$. Then the mild solution to
\begin{equation*}
-dp(t) = A'p(t)dt + \gamma(t)dt - q(t)dW(t), \qquad \quad p(T) = \xi
\end{equation*} 
satisfies the following a priori estimate $\forall t \in [0,T]$, for every $\beta >0$:
\begin{equation}\label{stimaclassica}
\E\int_0^T \left(|p(t)|^2_H + |q(t)|^2_{ \mathcal{L}_2(K;H)}	 \right)e^{\beta t}dt \leq C \left[  \E[ |\xi|_H^2e^{\beta T}] + \frac{2}{\beta} \E\int_0^T |\gamma(t)|^2_He^{\beta t}dt  \right] 
\end{equation}
for some constant $C$ depending on $T$ and the constants appearing in Hypothesis 1.
\end{lemma}

\begin{proof}
The solution to this equation was proved  in \cite{hu1991adapted}, Proposition 2.1, see also \cite {guatteri2005backward} for the estimates with exponential weights, Theorem 4.4 in particular  estimate $(4.13)$ and $(4.14)$. 
\end{proof}

Suppose first that the final cost is independent of the past, which is to say that in Hypothesis \ref{ipotesi} $\mu_h$ coincides with $\delta_T$,
the Dirac
measure in $T$. If the final cost is independent on the past, in the adjoint equation (\ref{ABSDEtrascinataMild}) the terminal condition is given at the final time $T$
and it is equal to $ D_xh(X(T))\in L^2_{\Fcal_T}(\Omega;H)$. Let us consider the following anticipated backward SPDE (ABSDE):
\begin{align}\label{eq:ABSPDEDiff}
-dp(t)& = A' p(t)\, dt  + \E^{\Fcal_t}\int_{-d}^0 D_xf(X(t),u(t-\theta))'p(t-\theta)\mu_f(d\theta)d t \nonumber\\ 
\
&+\E^{\Fcal_t} \int_{-d}^0 D_xl(X(t+\theta),u(t))\mu_l(d\theta)dt - q(t)\,dW(t)\\ 
p(T)&= D_xh(X(T)) \nonumber \\
p(t)&=0,  \ \text{for all } t \in ]T,T+d], \P-\text{a.s.},\   q(t)=0 \  \text{a.e. } t \in ]T,T+d], \P-\text{a.s.}.\nonumber  
\end{align}
We extend the notion of mild solutions to the anticipating case.
\begin{definition}
We say that equation \eqref{eq:ABSPDEDiff} admits a mild solution if there is a couple of
processes
$$(p,q) \in L^2_\mathcal{F}(\Omega\times [0,T+d];H)
\times L^2_\mathcal{F}(\Omega\times [0,T+d]; \mathcal{L}_2(K;H))$$
that solves the following equation, for all $t \in [0,T]$:
\begin{align}\label{eq:ABSPDE}
p(t) &= S(T-t)'D_xh(X(T))+ \int_t^T S(s-t)' \E^{\Fcal_s}\int_{-d}^0 D_xf(X(s),u(s-\theta))'p(s-\theta)\mu_f(d\theta)ds \\
&+ \int_t^T S(s-t)'\int_{-d}^0 D_xl(X(s+\theta),u(s))\mu_l(d\theta)ds - \int_t^T S(t-s)'q(s)dW(s)\nonumber
\end{align}
And $p(t)=0$ for all $t \in ]T,T+d], \P-$a.s., $ q(t)=0$ a.e. $t \in ]T,T+d], \P-$a.s..
\end{definition}
\begin{theorem}
Suppose that Hypothesis \ref{ipotesi} holds, with $\mu_h=\delta_T$. Then the anticipated backward SPDE \eqref{eq:ABSPDE} admits a unique mild solution 
$(p,q) \in  L^2_{\Fcal}(\Omega \times [0,T+d];H) \times  L^2_{\Fcal}(\Omega \times [0,T+d];\Lcal_2(K;H))$. 

Moreover we have that, for $\beta$ large enough:
\begin{equation}
\label{stimaantic}
\E\int_0^{T+d} 
\left( \frac{1}{2}|{p}(t)|^2_H + |q(t)|^2_{ \mathcal{L}_2(K;H)}	 \right)
e^{\beta t}dt \leq  C\Big[1+ \E\!\!\sup_{t\in [-d,T]}\!\!|X(t)|^{2(j\vee k)}  _H+ 
  \E \int_0^T\!\! |u(t)|^2 _U\, dt\Big]
\end{equation}
for some constant $C$ depending on $\beta, T$ and the constants appearing in Hypothesis 1.
\end{theorem}

\begin{proof}
In order to prove the existence of a solution to \eqref{eq:ABSPDE} we want to construct a contraction map, as in \cite{PengYang} 
in the finite dimensional case, but with some differences due to the infinite dimensional setting. Given a pair of processes 
$(y,z) \in  L^2_{\Fcal}(\Omega \times [0,T+d];H) \times  L^2_{\Fcal}(\Omega \times [0,T+d];\Lcal_2(K;H))$, we define the contraction map
$\Gamma: L^2_{\Fcal}(\Omega \times [0,T+d];H) \times  L^2_{\Fcal}(\Omega \times [0,T+d];\Lcal_2(K;H))
\to L^2_{\Fcal}(\Omega \times [0,T+d];H) \times  L^2_{\Fcal}(\Omega \times [0,T+d];\Lcal_2(K;H))$
using the mild formulation:
\begin{equation*}
\begin{system}
p(t) = S(T-t)' D_xh(X(T)) + \dis\int_t^T S(s-t)' \gamma(s,y_s)ds - \dis\int_t^T S(t-s)'q(s)dW(s)\\
p(t) = 0  \qquad   \forall  t \in ]T,T+d], \P- a.s.\\
q(t) = 0 \qquad a.e. \ t \in ]T,T+d], \P- a.s.\
\end{system}
\end{equation*}
where 
\begin{align*}\gamma(s, y_s)=  \E^{\Fcal_s}\int_{-d}^0 D_xf(X(s),u(s-\theta))'y(s-\theta)\mu_f(d\theta)   
+\int_{-d}^0 D_xl(X(s+\theta),u(s))\mu_l(d\theta)
\end{align*}
that belongs to $L^2_{\Fcal}(\Omega \times [0,T+d];H)$,  thanks to the assumptions on $f$ and $l$ and \eqref{stimafor}.
We define the map $\Gamma$ from $L^2_{\Fcal}(\Omega \times [0,T+d];H) \times  L^2_{\Fcal}(\Omega \times [0,T+d];\Lcal_2(K;H))$ to itself, such that $\Gamma(y,z) = (p,q)$ and $\Gamma(y',z') = (p',q')$, where $y,y'$ and $z,z'$ are arbitrary elements of the space just defined. Denote their differences by
\begin{equation*}
(\hat{y},\hat{z}) = (y-y', z-z'), \qquad (\hat{p},\hat{q}) = (p-p', q-q')  
\end{equation*} 
Using estimate in Lemma \ref{l:apriori.esitmate.backward}, according to the definition of $p,p'$, $q,q'$ we can say that
\begin{equation*}
\begin{split}
\E\int_0^{T+d} \left( |\hat{p}(t)|^2_H + |\hat{q}(t)|^2_{\Lcal_2(K;H)}\right)e^{\beta t}dt &\leq \frac{2}{\beta}
\E\int_0^T |\gamma(t,y_t) - \gamma(t,y'_t)|^2e^{\beta t}dt
\end{split}
\end{equation*}
The right hand side of the estimate can be rewritten using the special form of $\gamma$, then
\begin{equation*}
\begin{split}
& \E\int_0^{T+d} \left( |\hat{p}(t)|^2_H + |\hat{q}(t)|^2_{\Lcal_2(K;H)}\right)e^{\beta t}dt 
\leq \frac{2}{\beta} \E\int_0^T \big|\int_{-d}^0 D_xf(X(t),u(t-\theta))' \hat{y}(t - \theta)\mu_f(d\theta)\big|^2e^{\beta t}dt \\
&\leq \frac{2C_1}{\beta} \E\int_0^T \int_{-d}^0 |\hat{y}(t - \theta)|^2_H|\mu_f|(d\theta)e^{\beta t}dt 
\end{split}
\end{equation*}
After a change of variable $s = t- \theta$ we obtain 
\begin{equation}
\begin{split}
 \E\int_0^T \int_{-d}^0|\hat{y}(t-\theta)|^2_H|\mu_f|(d\theta)e^{\beta t}dt &
 \leq  \E\int_{0}^d \int_{-s}^0 |\hat{y}(s)|^2_He^{\beta (s+\theta)}|\mu_f|(d\theta)ds \\
&+ \E\int_{d}^{T} \int_{-d}^0 |\hat{y}(s)|^2_He^{\beta (s+\theta)}|\mu_f|(d\theta)ds  \\
&+ \E\int_{T}^{T+d} \int_{-d }^{0 \wedge (T-s)} |\hat{y}(s)|^2_He^{\beta (s+\theta)}|\mu_f|(d\theta)ds
\end{split}
\end{equation}
So that, there exists a constant $C$, that depends on $T$, such that
\begin{equation}
\frac{2}{\beta} \E\int_0^T \int_{-d}^0 |\hat{y}(t - \theta)|^2_H|\mu_f|(d\theta)e^{\beta t}dt \leq \frac{2C}{\beta}\E\int_0^{T+d}|\hat{y}(t)|_H^2e^{\beta t}dt
\end{equation}
We can choose $\beta$ such that the map $\Gamma$ is a strict contraction. By the Fixed Point Theorem we get existence and uniqueness of a mild solution.
Let us come to the estimate. Let $(p,q)$ be the fixed point solution, hence by \eqref{stimaclassica} we have, for a $\beta$ large enough: 
\begin{align*}
&\E\int_0^{T+d} \left( \frac{1}{2}|{p}(t)|^2_H + |q(t)|^2_{\Lcal_2(K;H)} \right)e^{\beta t}dt \\
&\leq  C \left[ 1+ e^{\beta T}  \E|D_xh(X(T))|^2_H+ \frac{2 e^{\beta T}(T+d)}{\beta}  \E\!\!\sup_{t\in [-d,T]}\!\!|X(t)|^{2j}
+  \frac{2 e^{\beta T}}{\beta} \E \int_0^T\!\! |u(t)|^2 \, dt  \right]\\  
& \leq   C\left[1+ e^{\beta T} \Big[ C_4 + \frac{2 (T+d)}{\beta} \Big] \E\!\!\sup_{t\in [-d,T]}\!\!|X(t)|^{2(j\vee k)} +  \frac{2 e^{\beta T}}{\beta} \E \int_0^T\!\! |u(t)|^2 \, dt. \right]
\end{align*}
Notice that the constant $C$  does not depend on $\beta$, since it comes from \eqref{stimaclassica}.
\end{proof}
 
In the following, we build an approximating ABSDE whose differential form  still makes sense and can approximate equation \eqref{ABSDEtrascinataMild}. Such approximating ABSDE
is obtained by a suitable approximation of $\mu_h$: the construction of this sequence of approximating measures $(\mu^n_h)_{n\geq 1}$
is given in the following Lemma.
\begin{lemma}\label{lemma:approx-measure}
 Let $\bar\mu$ be a finite  regular measure on $[T-d,T]$, such that $\bar \mu (\left\lbrace T \right\rbrace)=0$.
 There exists a sequence $(\bar\mu^n)_{n\geq 1}$ of finite regular measure on $[T-d,T]$, absolutely continuous with respect to
 $\lambda_{[T-d,T]}$, the Lebesgue measure on $[T-d,T]$, such that
 \begin{equation}\label{conv-a-mu}
  \bar\mu=w^*-\lim_{n\rightarrow\infty}\bar \mu^n,
 \end{equation}
that is for every $\phi\in C_b([T-d,T];\R)$
\begin{equation}\label{conv-a-mu:expl}
  \int_{T-d}^T \phi\,d\bar\mu=\lim_{n\rightarrow\infty}\int_{T-d}^T \phi\, d\bar\mu^n
 \end{equation}
\end{lemma}
\begin{proof} We notice that, denoted by $\calm$ the set of  regular probability measures on a bounded and closed interval $I$, 
then $\text{Extr}\calm=\left\lbrace \delta_x:x\in I\right\rbrace$. 
Moreover, by the Krein-Milman Theorem it turns out that $\calm=\bar{co}\text{Extr}\calm=\bar{co}\left\lbrace \delta_x:x\in I\right\rbrace$,
and so any probability measure in $\calm$ can be approximated by a linear convex combination of $\delta_x,\, x\in I$.

\noindent Our aim is to approximate a finite  regular measure $\bar\mu$
on the interval $I=[T-d,T]$. To this aim, there exist (at least) two positive, finite,  regular measures $\bar\mu^+$
and $\bar\mu^-$ such that $\forall\,A \in \calb([T-d,T])$,
$$
\bar\mu(A)=\bar\mu^+(A)-\bar\mu^-(A).
$$
Since  $\bar\mu^+$ and $\bar\mu^-$ are positive, finite,  regular measures on $[T-d,T]$,
it is possible to associate to  $\bar\mu^+$ and $\bar\mu^-$ two regular probability measures $\P^+ $ and $\P^-$ such that
\[
 \P^+=\frac{\bar\mu^+}{\bar\mu^+([T-d,T])},\;\P^-=\frac{\bar\mu^-}{\bar\mu^-([T-d,T])}.
\]
By what pointed out at the beginning of the proof, each of the two regular probability measures $\P^+ $ and $\P^-$
can be approximated by a linear convex combination of $\delta_x,\, x\in I$, namely
\begin{align*}
 &\bar\mu^+=w^*-\lim_{n\rightarrow\infty}\bar\mu^+([T-d,T])\sum_{k=1}^n\alpha_k\delta_{x_k} ,\; x_k\in I, \alpha_k\geq 0,\,\sum_{k=1}^n\alpha_k=1;\\ \nonumber
 &\bar\mu^-=w^*-\lim_{n\rightarrow\infty}\bar\mu^-([T-d,T]) \sum_{k=1}^n\beta_k\delta_{y_k},\; y_k\in I, \beta_k\geq 0,\,\sum_{k=1}^n\beta_k=1;
\end{align*}
and so 
\begin{equation*}
 \bar\mu=w^*-\lim_{n\rightarrow\infty}\left[\bar\mu^+([T-d,T])\sum_{k=1}^n\alpha_k\delta_{x_k}-
 \bar\mu^-([T-d,T])\sum_{k=1}^n\beta_k\delta_{y_k}\right],
 \end{equation*}
 where
\begin{equation*}
 x_k,\,y_k\in I, \alpha_k\geq 0,\,\sum_{k=1}^n\alpha_k=1 ,\, \beta_k\geq 0,\,\sum_{k=1}^n\beta_k=1.
\end{equation*}
This gives and approximation of $\bar \mu $ by means of discrete measures.
Now each $\delta$ measure can be approximated by means
of a measure which is absolutely continuous with respect to the Lebesgue measure.
Namely let us
set
$$
m_n=\min_{k=1,...,n-1}\left( x_{k+1}-x_k  \right),
$$
and let us consider the meausures $\lambda_{k,j},\, j<\frac{m_n}{2}$, absolutely continuous with respect to $\lambda$,
such that $\forall A \in \calb(I)$
$$
\lambda_{k,j}(A)=\int_A \frac{1}{j} \chi_{[x_k,x_k+j]}(x)\,dx.
$$
It turns out that, for each $k$,
\[
 \delta_k=w^*-\lim_{j\rightarrow\infty}\lambda_{k,j}.
\]
Moreover for every function $\xi\in C_b(I)$, by the mean value Theorem, with
$x_{k^*}\in[x_k,x_k+j],\, j<\frac{m_n}{2}$
\begin{align*}
 &\left|\sum_{k=1}^n\alpha_k\left[\int_I \xi(x)\,d\delta_{x_k}(x) - \int_I \xi(x)\lambda_{k,j}(x)(dx)\right]\right|
 =\left|\sum_{k=1}^n\alpha_k\left(\xi(x_k)-\xi(x_{k^*})\right)\right|\\ \nonumber
  &\qquad\leq\sum_{k=1}^n\alpha_k\omega\left(\vert x_k-x_{k^*}\vert\right)\leq \omega\left(m_n\right)\sum_{k=1}^n\alpha_k = \omega\left({m_n}\right)\rightarrow 0 \text{ as }n\rightarrow\infty, \nonumber
\end{align*}
where $\omega(\cdot)$ is the modulus of continuity of $\xi$.
The desired approximation result holds true.
\end{proof}
\begin{remark}
 Notice that the aim is the approximation of the measure $\mu_h$ in equation (\ref{ABSDEtrascinataMild}), in order to give sense
 to the differential form of such equation, and for this measure it can be $\mu_h(\left\lbrace T\right\rbrace)\neq 0 $. To the measure
 $\mu_h$ we can associate another measure $\bar \mu_h$ such that for any $A\in\calb ([T-d,T])$
 \begin{equation}\label{bar-mu}
   \bar\mu_h(A)=\mu_h(A\backslash\left\lbrace T\right\rbrace).
 \end{equation}
 Roughly speaking the measure $\bar \mu_h$ is obtained by the original measure $\mu_h$, by subtracting to $\mu_h$ its mass in
 $\left\lbrace T\right\rbrace$.
 Lemma \ref{lemma:approx-measure} ensures that there exists a sequence of measures $(\bar\mu^n_h)_{n\geq 1}$, on $[T-d,T]$,
 which are absolutely continuous with respect to the Lebesgue measure on $[T-d,T]$ and that are $w^*$-convergent to $\bar\mu_h$.
\end{remark}
We can define the following regular approximations of equation \eqref{ABSDEtrascinataMild} by approximating 
$\bar\mu_h $ obtained by $\mu_h$ in (\ref{bar-mu}).
We first introduce a preliminary equation
\begin{equation}\label{ABSDEtrascinata-approx-forz}
\begin{split}
p^n(t) &= \int_{t \vee (T-d)}^T\!\! S\left(s-t \right) ' D_x h(X(s)) \bar{\mu}^n_h(ds) \\
&+ \int_t^T S(s-t)' \Lambda(s)\, ds + \int_t^T S(s-t)' \, q^n(s)\,dW_s, \qquad t \in [0,T]\\
\end{split}
\end{equation}
where $\Lambda \in L^2_{\Fcal}(\Omega \times [0,T];H)$.
We can prove that:
\begin{lemma}
Assume Hypothesis 1 and that $\Lambda \in L^2_{\Fcal}(\Omega \times [0,T];H)$.
Then there exists a unique mild solution to \eqref{ABSDEtrascinata-approx-forz}. Moreover there exists a  positive constant $C$, depending only on $T$ and constants appearing in hypothesis 1, such that
for every $\beta >0$:
\begin{align}\label{stimaInter}
&\E\int_0^{T} \left( \frac{1}{2}|{p^n}(t)|^2_H + |q^n(t)|^2_{\mathcal{L}_2(K;H)}	 \right)e^{\beta t}dt &\\ \nonumber  &\leq  C\Big[\E \int_0^Te^{\beta t}\Big| \int_{t \vee (T-d)}^T S(s-t )' D_x h(X(s)) \bar{\mu}^n_h(ds)\Big|_H^2  \, dt \\ \nonumber&+ \frac{2}{\beta} \E\int_0^T |\Lambda(t)|^2_He^{\beta t}dt  +\E \int_0^Te^{\beta t}\Big| \int_{t \vee (T-d)}^T S(s-t)' K(s,t)\bar{\mu}^n_h(ds)\Big|_{\mathcal{L}_2(K;H)}^2\,dt \Big]
\end{align}
where $K\in  L^2_{\Fcal}(\Omega \times [0,T] \times [0,T];\mathcal{L}_2(K;H)) $ is such that for every $s \geq t\geq 0$:
\begin{equation}\label{repreFormula}
\E^{\mathcal{F}_t} D_x h(X(s))= D_xh(X(s)) - \int_t^s K(s,\theta)\, dW_{\theta}, \quad d\P \times d\theta  \text{ a.s.}
\end{equation}
\end{lemma}
\begin{proof}
Let us notice first of all that $\bar{\mu}^n_h (dt)= \bar{\zeta}^n_h(t)\,dt$, with $  \bar{\zeta}^n_h \in L^\infty2([T-d,T])$, thus the term $ D_x h(X(s)) \bar{\zeta}^n_h(s)$ can be treated as a ``normal'' forcing term and  lemma \ref{l:apriori.esitmate.backward} can be applied to get existence and uniqueness  of \eqref{ABSDEtrascinata-approx-forz}, for every $n$.
 
We concentrate on the estimate \eqref{stimaInter}, we consider $t > T-d$, the most critical case:
\begin{equation}
p^n(t)= \E^{\mathcal{F}_t} p^n(t)=  \E^{\mathcal{F}_t} \int_t^T S(s-t)' D_xh(X(s)) \bar{\mu}^n_h(ds) +
\E^{\mathcal{F}_t} \int_t^T S(s-t)' \Lambda(s) ds
\end{equation}
by standard estimates we get:
\begin{align*}
 \E \,  \int_0^T \frac{1}{2} |p^n(t)|_H^2 e^{\beta t} \, dt & \leq \int_0^T 
e^{\beta t} \E \Big| \int_{t}^T S(s-t)' D_x h(X(s)) \bar{\mu}^n_h(ds)\Big|_H^2  \, dt \\ 
& + \E  \int_0^T e^{\beta t} \Big ( \int_{t}^T e^{-\beta s}  \,ds \int_{t}^T e^{\beta s} | S(s-t)' 
\Lambda(s)|_H^2  \, ds \Big)\, dt \\&  \leq  \int_0^T 
e^{\beta t}\E \Big| \int_{t}^T S(s-t)' D_x h(X(s)) \bar{\mu}^n_h(ds)\Big|_H^2  \, dt +
\frac{M}{\beta}  \int_0^T e^{\beta t} |\Lambda(t)|_H^2  \, dt
\end{align*}
where the constant $M$ depends on $T$ and the constants appearing in Hypothesis 1.

Let us come to the martingale term $q^n$.
Following \cite{hu1991adapted}, we get that:
\begin{equation}
q^n(t)=\int_t^T S(s-t)'K(s,t)\bar{\mu}^n_h(ds) + \int_t^T S(s-t)'\tilde{K}(s,t) \, ds,  \quad d\P \times dt  \text{ a.s.}
\end{equation}
where $K$, $\tilde {K}$ are given by the Martingale Representation Theorem respectively in \eqref{repreFormula} and 
\begin{equation}\label{repreFormula1}
\E^{\mathcal{F}_t}\Lambda(s)= \Lambda(s) - \int_t^s \tilde{K}(s,\theta)\, dW_{\theta}, \quad d\P \times d\theta  \text{ a.s. }
\end{equation}
Thus, see also \cite{guatteri2005backward}, section 4, equation (4.14), we get that
\begin{align*}
\E \int_0^T e^{\beta t } |q^n(t)|^2_{\mathcal{L}_2(K;H)} \, dt \leq 2 \E \Big[\int_0^T 
e^{\beta t} \Big| \int_{t}^T S(s-t)'K(s,t) \bar{\mu}^n_h(ds)\Big|_{\mathcal{L}_2(K;H)}^2  \, dt + \frac{C}{\beta} \int_0^T e^{\beta t}|\Lambda(t)|^2 _H \, dt
\Big]
\end{align*}
where $C$ depends only on $ T$ and the constants appearing in Hypothesis 1.
The case when $ 0\leq t \leq T-d$ can be treated in the same way.
\end{proof}

\begin{proposition}\label{prop:approxABSDEtrascinata}
 Let Hypothesis 1 holds true, let $\bar\mu_h$ be defined by (\ref{bar-mu}), and let us consider $(\bar\mu^n_h)_{n\geq 0}$ the 
 $w^*$-approximations of $\bar\mu_h$, absolutely continuous with respect to the Lebesgue measure on $[T-d,T]$.
 Let us consider the approximating ABSDEs (of ``standard'' type):
 \begin{equation}\label{ABSDEtrascinata-approx-mu}
\begin{split}
p^n(t) &= \int_{t \vee (T-d)}^T\!\! S\left(s-t \right) ' D_x h(X(s)) \bar{\mu}^n_h(ds) \\
&+ \int_t^T S(s-t)' \, \E^{\Fcal_s}\!\int_{-d}^0 D_xf(X(s),u(s-\theta))'p^n(s-\theta)\mu_f(d\theta)\, ds \\
&+ \int_t^T S(s-t)'\, \E^{\Fcal_s}\!\int_{-d}^0 D_xl(X(s+\theta),u(s))\mu_l(d\theta)ds - 
\int_t^T S(t-s)'q^n(s)\, dW(s) \\ & + \dis S(T-t)'D_x h(X(T))\mu_h(\{T\}), \qquad t \in [0,T]\\
&p^n(t) = 0  \qquad   \forall  t \in ]T,T+d], \P- a.s.\quad
q^n(t) = 0 \qquad a.e. \ t \in ]T,T+d], \P- a.s.\
\end{split}
\end{equation}
Then the pair $(p^n,q^n)$, solution to (\ref{ABSDEtrascinata-approx-mu}), converges
in $L^2_{\mathcal{F}}(\Omega\times [0,T+d]; H)\times 
L^2_{\mathcal{F}}(\Omega\times [0,T+d];\mathcal{L}_2(K;H))$ to the pair $(p,q)$, mild solution to
(\ref{ABSDEtrascinataMild}). Moreover such solution is unique in $L^2_{\mathcal{F}}(\Omega\times [0,T+d]; H)\times 
L^2_{\mathcal{F}}(\Omega\times [0,T+d];\mathcal{L}_2(K;H))$ .
\end{proposition}
\begin{proof}Let us first prove that the sequence $(p^n,q^n)_n$ is a Cauchy sequence in
$L^2_{\mathcal{F}}(\Omega\times [0,T+d]; H)\times 
L^2_{\mathcal{F}}(\Omega\times [0,T+d];\mathcal{L}_2(K,H) )$.
The equation satisfied by $(p^n_t-p^k_t,q^n_t-q^k_t) $, $\,n,k\geq 1$, turns out to be an ABSDE:

 \begin{equation}\label{ABSDEtrascinata-approx-nk}
\begin{split}
(p^n-p^k)(t) &= \int_{t \vee (T-d)}^T\!\! S\left(s-t \right) ' D_x h(X(s)) \bar{\mu}^n_h(ds) - \int_{t \vee (T-d)}^T\!\! S\left(s- t \right) ' D_x h(X(s)) \bar{\mu}^k_h(ds)\\
&+ \int_t^T S(s-t)' \, \E^{\Fcal_s}\!\int_{-d}^0 D_xf(X(s),u(s-\theta))'(p^n-p^k)(s-\theta)\mu_f(d\theta)\, ds \\
&+ \int_t^T S(s-t)'(q^n-q^k)(s)\, dW(s) 
\end{split}
\end{equation}
Now we use  estimate \eqref{stimaInter} with $\Lambda(s)= \E^{\Fcal_s}\!\dis\int_{-d}^0 D_xf(X(s),u(s-\theta))'(p^n-p^k)(s-\theta)\mu_f(d\theta) $  to  get 
 \begin{align}\label{convergenza}
&\E\int_0^{T+d} \left( \frac{1}{2}|({p^n}-p^k)(t)|^2_H + |(q^n-q^k)(t)|^2_{\mathcal{L}_2(K;H)}	 \right)e^{\beta t}dt &\\ \nonumber 
&\leq  C\Big[ \int_0^T\!\!\!\!e^{\beta t}\E\Big| \int_{t \vee (T-d)}^T\!\!\!\! S(s-t )' 
D_x h(X(s))(\bar{\mu}^n_h-\bar{\mu}^{k}_h)(ds)\Big|_H^2  \, dt \\ &  \nonumber+\E \int_0^Te^{\beta t}\Big| \int_{t \vee (T-d)}^T S(s-t )' K(s,t)(\bar{\mu}^n_h-\bar{\mu}^{k}_h)(ds)\Big|_{\mathcal{L}_2(K;H)}^2\,dt \\ \nonumber&+ \frac{2}{\beta} \E\int_0^{T+d} |({p^n}-p^k)(t)|^2_H e^{\beta t}dt  \Big]
 \end{align}
Hence,  choosing $\beta$ such that $ \frac{2}{\beta}  \leq \frac{1}{4}$ we get that 
 \begin{align}\label{convergenza1}
&\E\int_0^{T+d} \left( \frac{1}{4}|({p^n}-p^k)(t)|^2_H + |(q^n-q^k)(t)|^2_{\mathcal{L}_2(K;H)}\right)e^{\beta t}dt 
&\\ \nonumber  &\leq  C\Big[ \E \int_0^T\!\!\!\!e^{\beta t}\Big| \int_{t \vee (T-d)}^T\!\!\!\! S(s-t)' D_x h(X(s))(\bar{\mu}^n_h-\bar{\mu}^{k}_h)(ds)\Big|_H^2  \, dt \\ &  \nonumber+\E \int_0^Te^{\beta t}\Big| \int_{t \vee (T-d)}^T S(s-t )' K(s,t)(\bar{\mu}^n_h-\bar{\mu}^{k}_h)(ds)\Big|_{\mathcal{L}_2(K;H)}^2\,dt  \Big]
 \end{align}
In order to prove that the sequence $(p^n, q^n)$ is a Cauchy sequence, we need to exploit the $w^* $-convergence of the sequence of measures $\bar \mu^n_h$ as $n \to \infty$.

Let us consider the first term, since $s \to S(s-t)' D_x h(X(s))$ is continuous $\P$-a.s. and uniformly bounded we have, by  the Dominated Convergence Theorem:
\begin{align}\label{limite}
\lim_{n,k \to +\infty}\E \int_0^T\!\!\!\!e^{\beta t}\Big|
\int_{t \vee (T-d)}^T\!\!\!\! S(s-t)' D_x h(X(s))(\bar{\mu}^n_h-\bar{\mu}^{k}_h)(ds)\Big|_H^2  \, dt =0.
 \end{align}
The second term requires  some work, since $s \to S(s-t)'K(s,t)$ is not continuous.
First of all we notice that by \eqref{repreFormula} we get, by hypothesis 
\begin{align*}
\E\int_{0}^s |K(s,t)|^2 _{\mathcal{L}_2(K;H)}\, dt  \leq 4 \E |D_xh (X(s))|^2 \leq 4 C. 
\end{align*}
Hence, we get
\begin{align}\label{stimaKtot}
\E\int_{T-d}^T\int_{0}^s 
|K(s,t)|^2_{\mathcal{L}_2(K;H)}\, dt \, d 
|\bar{\mu}_h|(ds)  \leq 4 C | \bar{\mu}_h|([T-d,T]).
\end{align}
where $\bar\mu_h$ is the weak limit of the measures $\bar{\mu}^n_h$, and 
 $ |\bar{\mu}_h|$ is the positive measure given by $\bar\mu_h^+ + \bar\mu_h^-$.

Now we introduce some processes $K^{\varepsilon}(s,t)$ such that $s\to K^{\varepsilon} (s,t) $ has a continuous version for all $t \in [0,T]$, 
and that can be chosen so that:
\begin{subequations}
\begin{align}\label{approx}
& |K^\varepsilon (s,t)|\leq |K(s,t) |, \quad \text{ a.e. } s,t \in [T-d,T]\times [0,T]   \text{ and }  \P- a.s.,   \\ 
\label{approx2}
&\lim_{\varepsilon \to 0}\E\int_{0}^T \int_{T-d}^T |K^\varepsilon (s,t)- K(s,t) |^2|\bar{\mu}_h|(ds) \, dt=0.
 \end{align}
\end{subequations}
Such properties are very well known for  $H=\R$, see for instance \cite{rudin1987real}, and can be easily extended to a separable Hilbert space going through a basis expansion. 

We can finally conclude noticing that:
\begin{align*}
&\E \int_0^Te^{\beta t}\Big| \int_{t \vee (T-d)}^T S(s-t )' K(s,t)(\bar{\mu}^n_h-\bar{\mu}^{k}_h)(ds)\Big|_{\mathcal{L}_2(K;H)}^2\,dt  \\ &\leq 
2\Big[ \E \int_0^Te^{\beta t}\Big| \int_{t \vee (T-d)}^T S(s-t )' K^\varepsilon(s,t)(\bar{\mu}^n_h-\bar{\mu}^{k}_h)(ds)\Big|_{\mathcal{L}_2(K;H)}^2\,dt  \\ &+
\E \int_0^Te^{\beta t}\Big| \int_{t \vee (T-d)}^T S(s-t )' (K(s,t)-K^\varepsilon(s,t)) (\bar{\mu}^n_h-\bar{\mu}^{k}_h)(ds)\Big|_{\mathcal{L}_2(K;H)}^2\,dt\Big]
\\ &\leq 
C\Big[ \E \int_0^Te^{\beta t}\Big| \int_{t \vee (T-d)}^T S(s-t )' K^\varepsilon(s,t)(\bar{\mu}^n_h-\bar{\mu}^{k}_h)(ds)\Big|_{\mathcal{L}_2(K;H)}^2\,dt  \\
&+ \E \int_0^T \int_{T-d}^T \Big| K(s,t)-K^\varepsilon(s,t)\Big|_{\mathcal{L}_2(K;H)}^2|\bar{\mu}_h|(ds)\, dt\Big].
\end{align*}
We let first $n$ and $k$ go to infinity to let the first term go to zero. Then, using \eqref{approx} and the fact that by the Banach-Steinhaus
Theorem
\[
 \vert \bar\mu^n_h\vert\leq \vert \bar\mu_h\vert,
\]
we let $\varepsilon$ go to zero and we get the convergence to $0$ of the whole expression. 
So $(p_n,q_n)$ is a Cauchy sequence in $L^2_{\mathcal{F}}(\Omega\times [0,T+d]; H)\times 
L^2_{\mathcal{F}}(\Omega\times [0,T+d];\mathcal{L}_2(K,H) )$, thus there exists a couple $(\bar{p},\bar{q})$ such that:
\begin{equation}
\E\int_0^{T+d}  |(p^n-\bar{p})(s)|^2 _H \, ds +  \E\int_0^{T+d}  |(q^n-\bar{q})(s)|^2 _{\mathcal{L}_2(K;H) }\, ds \to 0 \qquad \text{ as } n \to \infty.
\end{equation}
It remains to show the convergence inside the mild formulation of the equation to get that the limit $(\bar{p},\bar{q})$ fulfils equation \eqref{ABSDEtrascinataMild}.
Writing again the equation satisfied by the difference $p^n-p^k$, and recalling the measurability property typical of the solutions of backward equations, we get:
 \begin{equation}\label{ABSDEtrascinata-approx-nk-bis}
\begin{split}
\E^{\mathcal{F}_t} (p^n-p^k)(t) &=  \E^ {\mathcal{F}_t}\int_{t \vee (T-d)}^T\!\! S\left(s-t \right) ' D_x h(X(s)) (\bar{\mu}^n_h -\bar{\mu}^k_h)(ds)\\
&+ \E^ {\mathcal{F}_t}\int_t^T S(s-t)' \, \E^{\Fcal_s}\!\int_{-d}^0 D_xf(X(s),u(s-\theta))'(p^n-p^k)(s-\theta)\mu_f(d\theta)\, ds 
\end{split}
\end{equation}
Similarly to \eqref{limite}, by the Dominated Convergence Theorem, we have
\begin{equation}
\E\Big| \E^ {\mathcal{F}_t}\int_{t \vee (T-d)}^T\!\! S\left(s-t \right) ' D_x h(X(s)) (\bar{\mu}^n_h -\bar{\mu}^k_h)(ds) \Big|^2
 \to 0 \quad \text{ as } n,k \to +\infty.
\end{equation}
Moreover we get that for some constant $M$, which depends only on $T$ and the quantities appearing in  Hypothesis 1,  the following holds
 \begin{align*}
  &\E\Big| \E^{\mathcal{F}_t}\!\!\int_t^T \E^{\mathcal{F}_s}S(s-t)' \dis\int_{-d}^0 D_xf(X(s),u(s-\theta))'(p^{n}-{p}^k)(s-\theta)\mu_f(d\theta) ds\Big|^2
\\ & \leq M |\mu_f|^2([-d,0]) \E \int_{0}^{T+d}|(p^{n}-{p}^k)(r)|^2  dr \to 0 \quad \text{ as } n,k \to +\infty.
\end{align*}
 Thus $ \E|p^n(t) -p^k(t)|^2 \to 0$ as $k,n \to \infty $ for every $t \in[0,T]$, and 
 there exists a $p(t)$ such that  $ \E|p^n(t) -p(t)|^2 \to 0$ as $n \to \infty $ for every $t \in[0,T]$. It can be seen, again by the 
 dominated convergence Theorem, that also $ \dis\int_0^{T+d} \E|p^n(t) -p(t)|^2 \,dt \to 0$ as $n \to \infty $ so such limit must coincide with $\bar{p}(t)$ at least $\P$ a.s. for a.e. $t \in [T-d,T]$.
We  eventually get that $(\bar{p},\bar{q})$, choosing if necessary ${p}$ instead of $\bar{p}$ is a mild solution to equation \eqref{ABSDEtrascinataMild} since every term in \eqref{ABSDEtrascinata-approx-mu} must converge to its expected limit. 

Uniqueness is not a problem since as soon as we calculate the difference of two mild solutions, the data disappear and we can treat the equation as in lemma 
\ref{l:apriori.esitmate.backward}.

 \end{proof}

\section{Functional SMP: Necessary form}
Now we are able to prove a version of the SMP, in its necessary form, for the control problem with state equation and cost functional given
by \eqref{eq:state} and \eqref{eq:cost}, respectively. \\
Let $(\bar{X},\bar{u})$ be an optimal pair.
We start this section by recovering the form of the derivative of the cost functional. 
\begin{lemma}\label{l:expansion_cost}
The cost functional $J(\cdot)$ is Gateaux differentiable and the derivative has the following form
\begin{equation*}
\begin{split}
\frac{d}{d\rho} & J(\bar{u}(\cdot) +\rho v(\cdot))|_{\rho = 0} = \E \int_0^T  \int_{-d}^0 D_xl(\bar{X}(t+\theta), \bar{u}(t))Y(t+\theta) \mu_l(d\theta) dt \\
  &+ \E \int_0^T  \int_{-d}^0  D_ul(\bar{X}(t+\theta), \bar{u}(t) )v(t) \mu_l(d\theta) dt + 
  \E\int_{T-d}^T  D_xh\left(\bar{X}(\theta)\right) Y(\theta) \mu_h(d\theta)
\end{split}
\end{equation*}
\end{lemma}

\begin{proof}
Let $(\bar X,\bar u)$ be an optimal pair and let $w$ be another admissible control, set $ v= w-\bar u$
and $u^\rho=\bar u+\rho w$.  We can write the variation of the cost functional in the form
\begin{equation*}
\begin{split}
 0&\leq \frac{J(u^\rho(\cdot))-J(\bar u(\cdot))}{\rho}\\
 & =  \E \frac{1}{\rho}\int_0^T  \int_{-d}^0  \left[ l(X^\rho(t+\theta), u^\rho(t)) - l(\bar X(t+\theta),\bar u(t))\right]\mu_l(d\theta) dt \\
& + \E\frac{1}{\rho}\int_{T-d}^T \left[h(X^\rho(\theta))- h(\bar X(\theta))\right]\mu_h(d\theta)=I_1+I_2.
 \end{split}
 \end{equation*}
 Now
 \begin{equation*}
 \begin{split}
&I_1 = \E \frac{1}{\rho}\int_0^T  \int_{-d}^0 \left[  l( X^\rho(t+\theta), u^\rho(t)) - l(\bar X(t+\theta), u^{\rho}(t))\right]\mu_l(d\theta) dt\\
  &+\E \frac{1}{\rho}\int_0^T  \int_{-d}^0 \left[ l(\bar{X}(t+\theta), u^\rho(t))- l(\bar X(t+\theta),\bar u(t))\right]\mu_l(d\theta)dt \\
  &= \E \int_0^T  \int_{-d}^0 \int_0^1 \langle D_xl(\bar{X}(t+\theta) + \lambda(X^\rho(t+\theta) - \bar{X}(t+\theta)), u^\rho(t)),(\tilde{X}^\rho(t+\theta) + Y(t+\theta))\rangle d\lambda \mu_l(d\theta) dt \\
  &+ \E \int_0^T  \int_{-d}^0 \int_0^1\langle  D_ul(\bar{X}(t+\theta), \bar{u}(t) + \lambda \rho v(t) ),v(t) \rangle d\lambda \mu_l(d\theta) dt  
\end{split}
\end{equation*} 

\begin{equation*}
\begin{split}
I_2 &= \E\frac{1}{\rho}\int_{T-d}^T \left[h(X^\rho(\theta))- h(\bar X(\theta))\right]\mu_h(d\theta) \\
&= \E\int_{T-d}^T \int_0^1 \langle D_xh\left(\bar{X}(\theta) + \lambda (X^\rho(\theta) - \bar{X}(\theta))\right), (\tilde{X}^\rho(\theta) + Y(\theta))\rangle d\lambda \mu_h(d\theta)
\end{split}
\end{equation*}

From \eqref{lemma:convergence X tilde} we know that 
\[
\lim_{\rho \rightarrow 0}\E \sup_{t \in [0,T]}|\tilde{X}^{\rho}(t)|_H^2 = 0
\]
so that, sending $\rho$ to 0, we obtain the required equivalence
\begin{equation*}
\begin{split}
0 &\leq \frac{d}{d\rho}J(\bar{u}(\cdot) +\rho v(\cdot))|_{\rho = 0} = \E \int_0^T  \int_{-d}^0 \langle D_xl(\bar{X}(t+\theta)), \bar{u}(t)),Y(t+\theta) \rangle\mu_l(d\theta) dt \\
  &+ \E \int_0^T  \int_{-d}^0 \langle D_ul(\bar{X}(t+\theta), \bar{u}(t) ),v(t) \rangle\mu_l(d\theta) dt + \E\int_{T-d}^T  \langle D_xh\left(\bar{X}(\theta)\right) ,Y(\theta)\rangle \mu_h(d\theta)
\end{split}
\end{equation*}
\end{proof}
Now define the Hamiltonian associated to the system by setting $\mathcal{H}: [0,T] \times C([-d,0];H) \times U_c \times H \to \R$ as
\begin{equation}\label{eq:hamiltonian}
\begin{split}
\mathcal{H}(t,x,u,p) &=  \< \int_{-d}^0 f(x(\theta),u)\mu_f(d\theta), p \>_H + \int_{-d}^0 l(x(\theta),u)\mu_l(d\theta) \\
&= \<F(x,u),p\>_H  + L(t,x,u).
\end{split}
\end{equation}
where in the last equality we adopted the formalism introduced in Remark \ref{r:abstract}.

Let us state the stochastic maximum principle, where for the Definition of $\bar X_t$ we refer to (\ref{defdiX_t}).
\begin{theorem}
Let Hypothesis 1 be satisfied and suppose that $(\bar{X},\bar{u})$ is an optimal pair for the control problem. Then there exist a pair of processes $(p,q) \in L^2_\mathcal{F}(\Omega\times [0,T+d];H)\times L^2_\mathcal{F}(\Omega\times [0,T+d];\mathcal{L}_2(K;H))$ which are the solution to the ABSDE \eqref{ABSDEtrascinataMild} such that the following variational inequality holds
\begin{equation*}
\<\frac{\partial}{\partial u}\mathcal{H}(t,\bar{X}_t,\bar{u}(t), p(t),q(t)), w - \bar{u}(t)\> \geq 0 
\end{equation*}
for all $w \in U_c$, $\mP \times dt$- a.e., where $\mathcal{H}$ is the Hamiltonian function defined in \eqref{eq:hamiltonian}.
\end{theorem}

\begin{proof}
Firstly we study the duality between the first variation process $Y(\cdot)$ and the adjoint process $p(\cdot)$. Then we rewrite the duality using the variation of the cost functional. \\
As a first step we are going to prove the following duality formula
\begin{equation}\label{eq:duality}
\begin{split}
\E\<Y(T),D_x h(X(T))\mu_h(\{T\})\> &+  \E\int_0^T \int_{-d}^0 \<D_xl(X(t+\theta),u(t)),Y(t)\>\mu_l(d\theta)dt\\
&= \E\int_0^T \int_{-d}^0 \< D_uf(X(t+\theta),u(t))'p(t),v(t)\>\mu_f(d\theta) dt. \\
&-\E\dis\int_{T-d}^T \<Y(\theta), D_xh(X(\theta))\>\bar{\mu}_h(d\theta)
\end{split}
\end{equation}
which is the crucial relation from which we can formulate the SMP.
Let us recall that the dual equation \eqref{ABSDEtrascinataMild} associated to the system can not be written in differential form. Then, in order to obtain the variational inequality, we have to approximate the measure $\mu_h$ by means of the sequence $\bar{\mu}^n_h$, once we have subtracted its mass in $\lbrace T \rbrace$. What we obtain is equation \eqref{ABSDEtrascinata-approx-mu}, which we rewrite below for the reader's convenience
\begin{equation}
\begin{split}
p^n(t) &= \int_{t \vee (T-d)}^T\!\! S\left(s-t\right) ' D_x h(X(s)) \bar{\mu}^n_h(ds) \\
&+ \int_t^T S(s-t)' \, \E^{\Fcal_s}\!\int_{-d}^0 D_xf(X(s),u(s-\theta))'p^n(s-\theta)\mu_f(d\theta)\, ds \\ 
&+ \int_t^T S(s-t)'\, \E^{\Fcal_s}\!\int_{-d}^0  D_xl(s,X(s+\theta),u(s))\mu_l(d\theta)ds - \int_t^T S(t-s)'q^n(s)\, dW(s) \\ & + \dis S(T-t)'D_x h(X(T))\mu_h(\{T\}), \qquad t \in [0,T]\\
p^n(t)&=0,  \qquad \text{for all } t \in ]T,T+d], \P-\text{a.s.},\qquad   q^n(t)=0 \quad  \text{a.e. } t \in ]T,T+d], \P-\text{a.s.} .
\end{split}
\end{equation}

For all $k \in \mathbb{N}$, let $ J_k:= k(k-A)^{-1}$, the bounded operators such that $ \lim_{k \to +\infty} |J_k x-x|_H=0$, for all $ x \in H $.  We set $p^{n,k}(t):= J'_k p^n(t)$ and $ q^{n,k}(t):= J'_k q^n(t)$, hence by Dominated Convergence Theorem it can be proved:
\begin{align}\label{pnk}
\E\sup_{t \in [0,T+d]} |p^{n,k}(t)- p^{n}(t)|^2 _H \to 0, \quad  \text{ as } k \to 0; \\\nonumber
\E \int_{0}^{T+d}|q^{n,k}(t)- q^{n}(t)|^2 _{\mathcal{L}(K;H)}\, dt \to 0, \quad  \text{ as } k \to 0. 
\end{align}
Notice that the processes $p^{n,k}$ admit an It\^o differential:
\begin{equation}\label{eq:p^nk}
\left\{
\begin{array}{ll}
- dp^{n,k}(t) = \Big[ A'p^{n,k}(t) +   J'_k\displaystyle \int_{-d}^0  D_xf(X(t),u(t-\theta))'p^{n}(t-\theta)\mu_f(d\theta) \\ \\
\quad\quad\quad\quad \quad + J'_k \dis\int_{-d}^0  D_xl(t,X_t(\theta),u(t))\mu_l(d\theta)\Big] dt +J'_kD_xh(X(t))\bar{\zeta}^n_h(t) dt - q^{n,k}(t)dW(t),\\ \\
p^{n,k}(T) =  J'_k D_x h(X(T))\mu_h(\{T\}),&
\end{array}
\right.
\end{equation}
where we denoted by $\bar{\zeta}^n_h(\cdot)$ the Radon-Nykodim derivative of $\bar{\mu}^n_h$, with respect to the Lebesgue measure
on $[T-d,T]$.

Let us now consider the $k$-approximation of the first variation process  $Y^k (t)=: J_k Y (t)$, see \eqref{eq:first.variation}.
The equation for $Y^k$ is 
\begin{equation}
\begin{system}
\displaystyle \frac{dY^k}{dt}(t) = AY^k(t) +  J_k\int_{-d}^0  D_xf(Xt+(\theta),u(t))Y(t+\theta)\mu_f(d\theta) \\
\qquad \qquad +  J_k\int_{-d}^0 D_uf(X_t(\theta),u(t))\mu_f(d\theta)v(t), \\
Y^k(0) = 0.
\end{system}
\end{equation}
One can prove, evaluating the difference between $Y$ and $Y^k$ and using standard arguments, that:
\begin{equation}
\E\sup_{t \in [-d,T]} |Y^k(t)- Y(t)|^2 _H\to 0, \quad \text{ as } k \to 0
\end{equation}
Computing the It\^o formula for the product $d\<p^{n,k}(t),Y^k(t)\>$, integrating in time and taking expectation we end up with
\begin{equation}\label{ito_duality}
\begin{split}
\E\<Y^k(T),p^{n,k}(T)\>_H &=  \E\int_0^T  J_k\< \int_{-d}^0 D_xf(X(t+\theta),u(t))Y(t+\theta)\mu_f(d\theta), p^{n,k}(t)\>_H dt \\
&+ \E\int_0^TJ_k \< \int_{-d}^0 D_uf(X(t+\theta),u(t))\mu_f(d\theta)v(t), p^{n,k}(t)\>_H dt \\
&- \E\int_0^T \<Y^k(t),J'_k\int_{-d}^0 D_xl(X(t+\theta),u(t))\mu_l(d\theta)\>_Hdt\\
&-\E\int_0^T \<Y^k(t), J'_k \int_{-d}^0 D_xf(X(t),u(t-\theta))'p^{n}(t-\theta)\mu_f(d\theta) \>_H dt \\
&-\E\dis\int_{T-d}^T\<Y^{k}(\theta), J'_k D_xh(X(\theta))\bar{\zeta}^n_h(\theta)\>d\theta
\end{split}
\end{equation}
Now we let $k$ tend to $\infty$ to get
the following duality relation
\begin{equation}\label{eq:duality_approx}
\begin{split}
&\E\<Y(T),D_x h(X(T))\mu_h(\{T\})\> +  \E\int_0^T \int_{-d}^0 \<D_xl(X(+\theta),u(t)),Y(t)\>\mu_l(d\theta)dt\\
&= \E\int_0^T \int_{-d}^0 \< D_uf(X(t+\theta),u(t))'p^{n}(t),v(t)\>\mu_f(d\theta) dt -\E\dis\int_{T-d}^T \<Y(\theta), D_xh(X(\theta))\>\bar{\mu}_h^n(d\theta)
\end{split}
\end{equation}
In order to obtain the required duality formula \eqref{eq:duality} we have to pass to the limit as $n\to \infty$ in \eqref{eq:duality_approx}. This can be  done using the result on convergence of measures stated in Lemma \ref{lemma:approx-measure} and the convergence results proved in Proposition \ref{prop:approxABSDEtrascinata}. 

We thus end up with \eqref{eq:duality}, and recalling that:
 \begin{equation*}\label{dec:mu_h}
 \mu_h= \bar \mu_h + \mu_h(\left\lbrace T\right\rbrace)\delta_T
 \end{equation*}
we can write \eqref{eq:duality} as:
\begin{equation}\label{eq:duality_lim}
\begin{split}
&  \E\dis\int_{T-d}^T \<Y(\theta), D_xh(X(\theta))\>{\mu}_h(d\theta) +
\E\int_0^T \int_{-d}^0 \<D_xl(X(t+\theta),u(t)),Y(t)\>\mu_l(d\theta)dt\\
&= \E\int_0^T \int_{-d}^0 \< D_uf(X(t+\theta),u(t))'p(t),v(t)\>\mu_f(d\theta) dt.
\end{split}
\end{equation}
On the other side, from lemma \ref{l:expansion_cost}, we know that
\begin{equation}\label{derivata}
\begin{split}
0 &\leq \frac{d}{d\rho}J(\bar{u}(\cdot) +\rho v(\cdot))|_{\rho = 0} = \E \int_0^T  \int_{-d}^0 \langle D_xl(t, \bar{X}(+\theta)), \bar{u}_t),Y_t(\theta)\rangle \mu_l(d\theta) dt \\
  &+ \E \int_0^T  \int_{-d}^0  \langle D_ul(t, \bar{X}(t+\theta), \bar{u}_t ),v_t \rangle\mu_l(d\theta) dt + \E\int_{T-d}^T  \langle D_xh\left(\bar{X}(\theta)\right), Y(\theta) \rangle\mu_h(d\theta)
\end{split}
\end{equation}
So substituting  \eqref{eq:duality_lim} in  \eqref{derivata}  we eventually get
\begin{equation}
\begin{split}
0 &\leq \E\int_0^T \int_{-d}^0 \< D_uf(\bar{X}(t+\theta),\bar{u}(t))'p(t),v(t)\> \mu_f(d\theta)dt + \E \int_0^T  \int_{-d}^0  \<D_ul(t, \bar{X}(t+\theta), \bar{u}_t ),v_t \>\mu_l(d\theta) dt \\
&= \E\int_0^T \<\frac{\partial}{\partial u}H(t,\bar{X}_t,\bar{u}(t), p(t),q(t)), v(t) \>dt
\end{split}
\end{equation} 
from which the required result holds true.
\end{proof}

\bibliography{mybib}
\bibliographystyle{plain}

\end{document}